\documentclass[a4paper,reqno]{amsart}
\usepackage{amssymb}
\usepackage{amsmath}
\usepackage{amsthm}
\usepackage{graphicx}
\usepackage{texdraw}
\usepackage{amsfonts}
\usepackage{hyperref}
\usepackage{caption}
\usepackage{color}
\usepackage{epsfig}
\usepackage{calligra}
\usepackage{mathtools}

 \allowdisplaybreaks \numberwithin{equation}{section}
\begingroup
\theoremstyle{plain}
\newtheorem{theorem}{Theorem}[section]
\newtheorem{proposition}[theorem]{Proposition}
\newtheorem{lemma}[theorem]{Lemma}

\theoremstyle{definition}
\newtheorem{definition}[theorem]{Definition}
\newtheorem{remark}[theorem]{Remark}

\endgroup

\def \de {\mathrm{d}}
\def \e {\varepsilon}

\def \A {\mathcal A}
\def \r {r_{_{2N+1}}}

\title[A Blaschke-Lebesgue Theorem for the Cheeger constant]{A Blaschke-Lebesgue Theorem for the Cheeger constant}

\author{A. Henrot, I. Lucardesi}

\begin{document}

\begin{abstract}
In this paper we prove a new extremal property of the Reuleaux triangle: it maximizes
the Cheeger constant among all bodies of (same) constant width. The proof relies on
a fine analysis of the optimality conditions satisfied by an optimal Reuleaux polygon
together with an explicit upper bound for the inradius of the optimal domain.
As a possible perspective, we conjecture that this maximal property of the Reuleaux
triangle holds for the first eigenvalue of the $p$-Laplacian for any $p\in (1,+\infty)$
(the current paper covers the case $p=1$ whereas the case $p=+\infty$ was already known).
\end{abstract}

\maketitle

\medskip



\section{Introduction}
Bodies of constant width (also named after L. Euler {\it orbiforms}) have attracted much attention in the mathematical community along the last centuries. 
Several surveys have been devoted to these objects, and contain an abundant literature.
We refer notably to a chapter in Bonnesen-Fenchel's
famous book \cite{BF}, a survey by Chakerian-Groemer in the book ``Convexity and its applications" \cite{CG}, and the recent book by Martini-Montejano-Olivaros \cite{MMO}.
In the plane, two bodies of constant width
play a particular role: the disk, of course, and the Reuleaux triangle (obtained by
drawing arcs of circle from each vertex of an equilateral triangle between the other two vertices). If all plane bodies of constant width have the same perimeter (this is Barbier's
Theorem), they do not have the same area and the two extreme sets are precisely the disk
(with maximal area by the isoperimetric inequality) and the Reuleaux triangle (with
minimal area).
This last result is the famous Blaschke-Lebesgue Theorem, see \cite{Bla} for the proof of W. Blaschke or \cite{KM} for a more modern exposition, \cite{Le} for the original proof of H. Lebesgue,
and \cite{BF}, where this proof is reproduced. Let us mention that many
other proofs with very different flavours (more geometric or more analytic)
appeared later, for example in \cite{Be}, \cite{CCG},
\cite{Eg}, \cite{Ga}, and \cite{Ha}. The disk and the Reuleaux triangle share this extremal
properties for other geometric functionals like {\it the inradius} and {\it the circumradius}, 
in particular the Reuleaux triangle minimizes the inradius among all bodies
of constant width, see e.g. \cite{BF} or \cite{CG}.

We believe that these extremal properties of the disk and the Reuleaux triangle hold
for more complicated functionals. In particular in Section \ref{conc}, we explain why we think
that the Reuleaux triangle maximizes the first eigenvalue of the $p$-Laplacian (with
Dirichlet boundary condition) for any $p$, $1\leq p\leq +\infty$. Note that it is well known
that the disk (or the ball in any dimension) minimizes this eigenvalue, for any $p$
and the proof is done by spherical rearrangement.

The aim of this paper is to make a first step in this direction by proving that the
Reuleaux triangle maximizes the Cheeger constant among all bodies of constant width. Indeed,
the Cheeger constant (defined below) can also be seen as the first eigenvalue of the
$1$-Laplacian, see \cite{KaFr}.

The Cheeger constant of a bounded plane domain $\Omega$ is defined as
\begin{equation}\label{ch1}
h(\Omega)=\min_{E\subset \Omega} \frac{P(E)}{|E|}
\end{equation}
where $P(E)$ is the perimeter of $E$ (defined as the perimeter in the sense of De Giorgi
for measurable sets) and $|E|$ is the area of $E$. In \eqref{ch1}, the minimum is achieved
as soon as $\Omega$ has a Lipschitz boundary. A set $E$
which realizes this minimum is called a {\it Cheeger set} of $\Omega$ and we denote it by $C_\Omega$. This notion, introduced by Jeff Cheeger in \cite{Che} (to obtain a geometric
lower bound for the first eigenvalue of the Laplacian), has extensively received
attention in the last decades. For an
introductory survey on the Cheeger problem we refer for example to \cite{Pa1}.
In general the Cheeger set is not unique, but it is unique if $\Omega$ is convex,
see \cite{AlCa}. 
Moreover, for convex planar domains, there is a nice characterization 
of the Cheeger constant and the Cheeger set, see e.g. Lachand-Robert and Kawohl \cite{KLR}: the Cheeger constant reads
\begin{equation}\label{defR}
h(\Omega)=\frac{1}{R(\Omega)},\quad \hbox{where $R(\Omega)$ satisfies }\  |\Omega_{-R}|=\pi R^2,
\end{equation}
where $\partial \Omega_{-R}$ is the inner parallel set to $\partial \Omega$ at distance $R$, and the Cheeger set is $C_\Omega=\Omega_{-R(\Omega)}+B_{R(\Omega)}$ (the Minkowski sum of $\Omega_{-R(\Omega)}$ and the disk of radius $R(\Omega)$).

Therefore, the main result of this paper is
\begin{theorem}\label{maintheo}
The Reuleaux triangle maximizes the Cheeger constant in the class of plane bodies
of constant width.
In other words, for any body $\Omega$ of constant width
\begin{equation}\label{ineq1}
h(\Omega)\leq h(\mathbb{T})
\end{equation}
where $\mathbb{T}$ is the Reuleaux triangle of same width.
\end{theorem}
Our strategy of the proof is as follows. Without loss of generality, we work with bodies of width 1. First of all, we look at this maximization problem
in the restricted class of Reuleaux polygons (with a number of sides less than $2N+1$). We will then generalize the result, exploiting the density of the Reuleaux polygons in the class of bodies with constant width.
We begin with a simple observation on the inradius of the optimal domain: it must be small, more precisely, smaller than $r_0\simeq 0.4305$. Note that the minimal value, obtained by the Reuleaux triangle, is $r_{min}=1-1/\sqrt{3}\simeq 0.4226$.
The key point to get such a precise estimate is the explicit computation of the minimal
area of a body of constant width enclosed in a given annulus, that we obtained in a
recent paper, see the Appendix and reference \cite{HL}.

Now, in the class of Reuleaux polygons, after having proved existence of
a maximizer, we obtain optimality conditions, thanks to the so-called {\it shape derivative}.
For that purpose, we consider only a particular kind of perturbations allowing us to
stay in the same class. These perturbations may be defined for any Reuleaux polygon
(except the Reuleaux triangle) and have been used by W. Blaschke in his proof of the Blaschke-Lebesgue
Theorem. They consist in sliding one vertex on its arc, moving that way
three corresponding arcs of the polygon in order to respect the constant width 
condition and letting all the other arcs unchanged. The optimality condition
we get is rather complicated, but it allows us to prove, through
a precise analysis of the functions involved, that the optimal domain has arcs
with very similar lengths: in Theorem \ref{theolengths} we give an estimate of
the ratio of the lengths of two consecutive arcs that happens to be close to 1.
To conclude, we are able to use this property of the lengths to prove that the inradius
of such Reuleaux polygon must be larger than $r_0$, first with a general proof in
the case $N\geq 7$, then for all the remaining values of $N=2,3,4,5,6$ by a simple analysis. This proves that the optimal
Reuleaux polygon cannot have more than 3 sides.

\medskip
In this paper, we define $\mathcal{B}^1$ as the class of plane bodies of constant width $1$ and $\mathcal{B}_N^1$ as the subclass
of Reuleaux polygon with (at most) $2N+1$ sides.
Throughout the paper we will always take the origin at the center of the incircle.

\section{Existence and a first optimality condition}
\subsection{Existence}
First of all, we show that the functional $h$ is bounded above in $\mathcal B^1$
by explicit bounds.

A first upper bound comes from two classical theorems: the Barbier Theorem (see, e.g. \cite{CG}) and the Blaschke-Lebesgue Theorem (see, e.g. \cite{Bla}). The former states that the perimeter of any plane body of constant width $1$ is $\pi$, the latter asserts that the Reuleaux triangle minimizes the area among plane bodies of constant width. By definition of $h$, we immediately get
\begin{equation}\label{harea}
h(\Omega)\leq \frac{\pi}{|\Omega|} \leq \frac{\pi}{|\mathbb{T}|}=\frac{2\pi}{\pi-\sqrt{3}}\sim 4.4576.
\end{equation}

Another possible strategy to get the boundedness is to exploit the monotonicity of $h$ with respect to the inclusion, together with the fact that $\mathbb T$ minimizes the inradius $r(\Omega)$ in $\mathcal B^1$ (see, e.g. \cite{BF}):
\begin{equation}\label{hrho}
h(\Omega) \leq h(B(0,r(\Omega))) = \frac{2}{r(\Omega)}\leq \frac{2}{r(\mathbb{T})}=\frac{2}{1-1/\sqrt{3}}\sim 4.732.
\end{equation}

\begin{proposition}\label{exist} The functional $h$ admits a maximizer in $\mathcal B^1$.
\end{proposition}
\begin{proof}
In \eqref{harea} (or \eqref{hrho}) we have shown that $h$ is bounded above in the class. Therefore, its supremum is finite. Let $\Omega_n$ be a maximizing sequence. Since the elements of $\mathcal B^1$ are convex bodies with prescribed constant width, we infer that they can all be enclosed into a compact set. Therefore, by Blaschke selection theorem, up to a subsequence (not relabeled), $\Omega_n\to \Omega^*$ with respect to the Hausdorff metric, for some convex body $\Omega^*$. Now it is classical that the class $\mathcal B^1$ is closed for the
Hausdorff metric (Hausdorff convergence is equivalent to uniform convergence of the support functions), thus $\Omega^*\in \mathcal B^1$. To conclude, we exploit the continuity of $h$ with respect to the Hausodrff metric. This is proved, e.g., in \cite[Proposition 3.1]{Pa2}.   
\end{proof}

Actually, the same existence result can be proved in the subclass $\mathcal B_N^1$ of Reuleaux polygons with at most $2N+1$ sides.

\begin{proposition}\label{existN} For every $N\in \mathbb N$, the functional $h$ admits a maximizer in $\mathcal B^1_N$.
\end{proposition}
\begin{proof} Arguing as in the proof of Proposition \ref{exist}, the statement follows by combining the boundedness of $h$ from above, the compactness of $\mathcal B^1_N$ with respect to the Hausdorff metric (see \cite[Proposition 2.2]{KM}), and the continuity of $h$ with respect to the Hausdorff metric.
\end{proof}

\subsection{The Cheeger constant of a Reuleaux triangle}
In this paragraph we compute $h(\mathbb T)$ using the implicit formula \eqref{defR}. We recall that the boundary of the Reuleaux triangle is formed by three arcs of circle of radius 1 and arc length $\pi/3$, centered at three boundary points $P_1$, $P_2$, and $P_3$. Without loss of generality, we choose the orientation in such a way that
\begin{equation}\label{P123}
P_1=\frac{1}{\sqrt{3}}\,e^{i 11 \pi /6},\quad P_2=\frac{1}{\sqrt{3}}\,e^{i \pi /2},\quad P_3= \frac{1}{\sqrt{3}}\,e^{i 7 \pi /6}.
\end{equation}
Given an arbitrary $0<R<1$, the boundary of the inner parallel set $\Omega_{-R}$ is made of three arcs of circle, centered at the $P_i$, with radius $1-R$. They meet at three points $Q_i$, $i=1,2,3$, which, by symmetry, lie on the segments $P_iO$, being $O$ the origin.

\begin{figure}[h]                                             
\begin{center}                                                
{\includegraphics[height=4truecm] {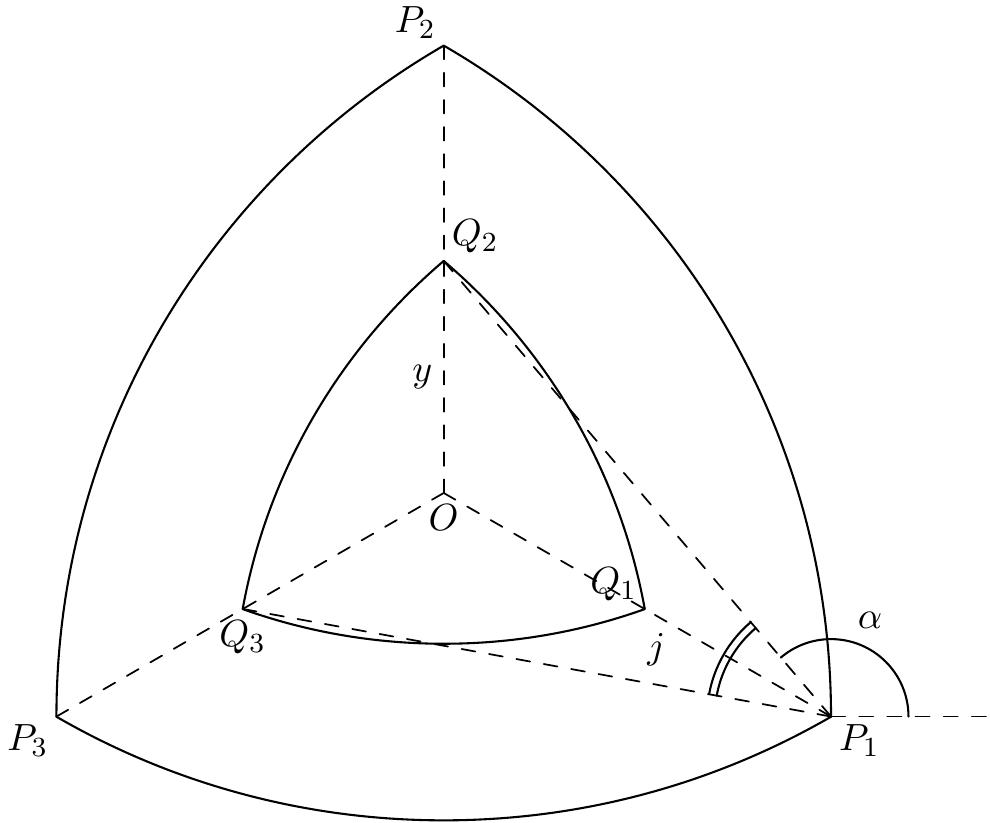}   \quad \quad   \quad   \includegraphics[height=4truecm] {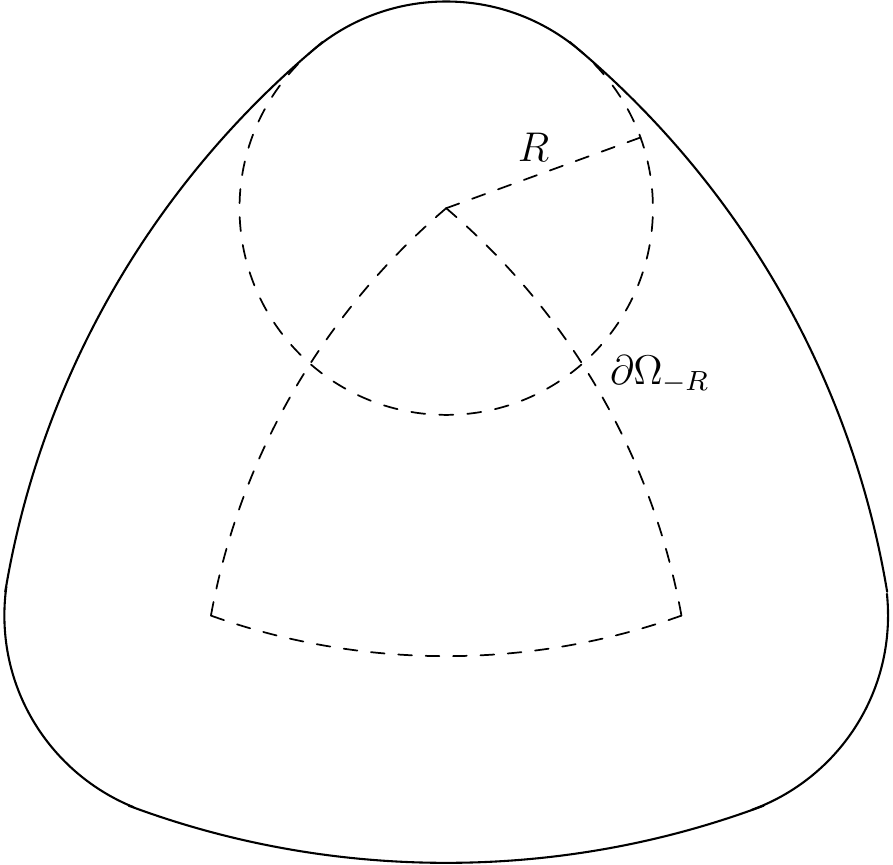}     }                                    
\end{center}                                                  
\caption{{\it Left: the Reuleaux triangle and an inner parallel set. Right: the Cheeger set of the Reuleaux triangle and the inner parallel set.}}\label{fig-parallel3} 
\end{figure}  

In order to determine the area of the inner parallel set, we need to compute the following objects: the angle $\alpha$ such that  $Q_2=P_1+(1-R) e^{i\alpha}$, the distance $y:=|OQ_2|$, and the angle $j:=\widehat{Q_2P_1Q_3}$ (see also Fig. \ref{fig-parallel3}). Recalling formulas \eqref{P123} and imposing that the horizontal coordinate of $Q_2$ is zero, we get
$$
\alpha=\arccos \left(- \frac{1}{2(1-R)}\right).
$$
Similarly, evaluating the vertical component of $Q_2$, we obtain
$$
y=(1-R)\sin\alpha -\frac{1}{2\sqrt{3}}\,. 
$$
Finally, it is immediate to check that $j=2(5\pi/6-\alpha)$.
Let us now compute the area. Connecting each $Q_i$ with the origin, the inner parallel set $\Omega_{-R}$ is divided into three parts of equal area, and we have
$$
|\Omega_{-R}|=\frac{3}{2} \left[ \frac{\sqrt{3}}{2}y^2 + (1-R)^2(j - \sin j)\right].
$$
Imposing \eqref{defR}, namely that $|\Omega_{-R}|=\pi R^2$, we find
\begin{equation}\label{Rtriangle}
0.22802 \leq R=R(\mathbb{T}) \leq 0.22803,
\end{equation}
implying
\begin{equation}\label{htriangle}
h(\mathbb{T})\geq 4.3853.
\end{equation}

\subsection{A first optimality condition}
The knowledge of $r(\mathbb T)$ and the computation of $h(\mathbb T)$ and $R(\mathbb T)$ allow us to get some necessary conditions on the values of the functionals $r$ and $R$ for maximizers.

\begin{proposition}\label{propr0}
Let $\Omega^*$ be a maximizer for $h$ in $\mathcal B^1$. Then
\begin{eqnarray}
0.21132 \leq \frac{r(\mathbb T)}{2} \leq &\!\!\!\! R(\Omega^*)\!\!\!\!&\leq R(\mathbb T) \leq 0.22803, \label{estimatesR}
\\
\smallskip\notag
\\
0.4226 \leq r(\mathbb T) \leq & \!\!\!\!r(\Omega^*)\!\!\!\!&\leq r_0:=0.4302. \label{r0}
\end{eqnarray} 
\end{proposition}
\begin{proof}
Let us start with $R$. By definition, $R(\Omega^*)=1/h(\Omega^*)\leq 1/h(\mathbb T) = R(\mathbb T)$. On the other hand, exploiting \eqref{hrho}, we get $R(\Omega^*)\geq r(\mathbb T)/2=(1-1/\sqrt{3})/2$. These inequalities, together with \eqref{Rtriangle}, prove \eqref{estimatesR}.

As already mentioned, the proof of the minimality of $\mathbb T$ for the inradius can be found in \cite{BF}, in particular $r(\mathbb T)\leq r(\Omega^*)$. In order to prove the upper bound for $r(\Omega^*)$, we introduce the auxiliary function 
\begin{equation}\label{A}
\begin{array}{lll}
\mathcal A:[1-1/\sqrt{3}; 1/2] & \longrightarrow \mathbb R^+
\\
&r \mapsto \mathcal A(r):=\min \left\{|\Omega|\ :\ \Omega\in \mathcal B^1,\ r(\Omega)=r \right\}.
\end{array}
\end{equation}
In other words, $\mathcal A$ associates to $r$ the minimal area of a shape in $\mathcal B^1$ with prescribed inradius. Note that the endpoints of the domain of $\mathcal A$ are the minimal and maximal inradius of shapes in $\mathcal B^1$. The properties of $\mathcal A$ and of the optimal shapes are investigated in \cite{HL}. For the benefit of the reader, the main facts are gathered in the Appendix, in the last section of the paper.

In view of definition \eqref{A} of $\mathcal A$, for every shape $\Omega$ in the class, we have $|\Omega| \geq \mathcal A(r(\Omega))$, so that, arguing as in \eqref{harea}, 
\begin{equation}\label{a1}
h(\Omega)\leq \frac{\pi}{\mathcal A(r(\Omega))}.
\end{equation}
On the other hand, for $\Omega^*$ maximizer, there holds $h(\Omega^*)\geq h(\mathbb T)$. This fact, combined with \eqref{a1}, gives 
$$
\mathcal A(r(\Omega^*))\leq \frac{\pi}{h(\mathbb T)}.
$$
Since $\mathcal A$ is strictly increasing, we infer that 
$$
r(\Omega^*) \leq \mathcal A^{-1} \left( \frac{\pi}{h(\mathbb T)}\right) <r_0:=0.4302,
$$
concluding the proof.
\end{proof}

\section{Optimality conditions in the class of Reuleaux polygons}
In this section we write a family of optimality conditions in the class of Reuleaux polygons, namely for the study of the maximization of $h$ in $\mathcal B_N^1$. To this aim, we need to fix some definitions. 

\subsection{Reuleaux polygons}

The boundary of a Reuleaux polygon $\Omega$ of width 1 is made of an odd number of arcs of radius 1, centered at boundary points $P_k$, $k=1,\ldots, 2N+1$, for some $N\in \mathbb N$. Notice that in this case $\Omega \in \mathcal B^1_M$ for every $M\geq N$. The boundary arc centered at $P_k$ is denoted by $\gamma_k$ and is parametrized by
\begin{equation}\label{param}
\gamma_k:=\{P_k + e^{is}\ :\ s\in [\alpha_k, \beta_k]\},
\end{equation}
for some pair of angles $\alpha_k, \beta_k$. We identify here the complex number $e^{is}$ with the point $(\cos s, \sin s)\in \mathbb R^2$. 
For brevity, we set
\begin{equation}\label{defj}
j_k:=\mathcal H^1(\gamma_k)\quad \mbox{(the length of $\gamma_k$)}.
\end{equation}
The vertexes are ordered as follows: the subsequent and previous points of $P_k$ are
$$
P_{k+1}=P_k + e^{i\alpha_k}\quad \hbox{and}\quad P_{k-1}=P_{k} + e^{i\beta_k},
$$
respectively. Accordingly, the angles satisfy
$$
\beta_{k+1}= \alpha_k +\pi \,\quad \hbox{mod }2\pi.
$$

The concatenation of the parametrizations of the arcs provides a parametrization of the boundary of the Reuleaux polygon in counter clockwise sense: the order is $\gamma_{2N+1}$, $\gamma_{2N-1}$, $\ldots$, $\gamma_{1}$, $\gamma_{2N}$, $\gamma_{2N-2}$, $\ldots$, $\gamma_2$, namely first the arcs with odd label followed by the arcs with even label, see e.g., Fig. \ref{figR}.

\begin{figure}[h]                                             
\begin{center}                                                
{\includegraphics[height=4.5truecm] {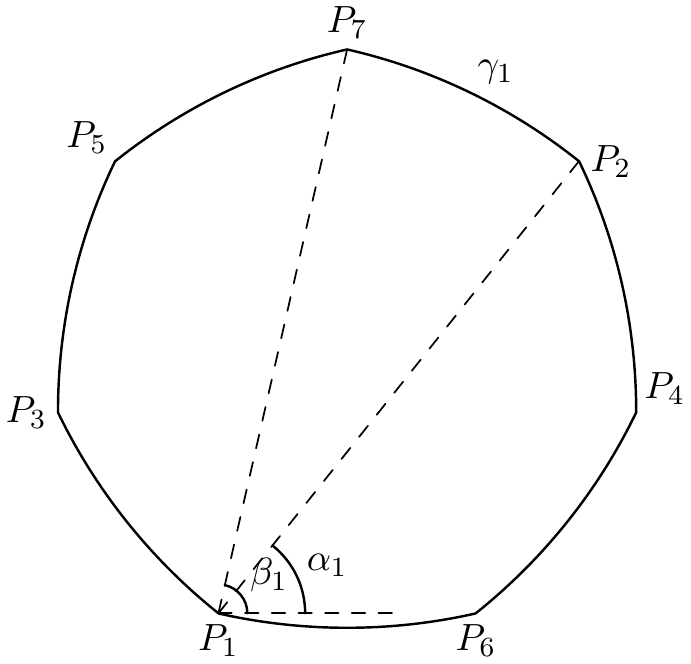}}                                                
\end{center}                                                  
\caption{{\it Notation of vertexes, arcs, and angles for the parametrization of a Reuleaux heptagon.}}\label{figR} 
\end{figure}

\subsection{The Cheeger set of a Reuleaux polygon}\label{ssR}
Let $\Omega$ be a Reuleaux polygon. According to \cite{KLR}, the boundary of the Cheeger set $C_\Omega$ is the union of a (non empty) portion of $\partial \Omega$ and arcs of circle of radius $R:=R(\Omega)$. Moreover, the arcs of circle meet $\partial \Omega$ tangentially. 

In view of the geometry of $\Omega$, we infer that the intersection $\partial C_\Omega \cap \partial \Omega$ is the union of arcs of circle of radius 1 of the form $\gamma_\ell':=\partial C_\Omega \cap \gamma_\ell$. We parametrize them as follows: 
\begin{equation}\label{arcprime}
\gamma_\ell'= \{P_\ell+ e^{is}\ :\ s\in [\alpha_\ell', \beta_\ell']\},
\end{equation}
for suitable $\ell$s and $\alpha_\ell\leq \alpha_\ell'\leq \beta_\ell'\leq \beta_\ell$.

Notice that, a priori, there might be an index for which $\gamma_\ell'=\emptyset$.

\medskip

We now show that the fact that the ``free part'' of the boundary of $C_\Omega$ meets $\partial \Omega$ tangentially entails a relation among the contact angles, the lengths of the arcs, and $R$. More precisely, let us assume that $\partial C_\Omega$ intersects two consecutive arcs: $\gamma_\ell$ and $\gamma_{\ell-2}$. In a neighborhood of their common point $P_{\ell-1}$, the boundary of the Cheeger set is the concatenation of $\gamma_\ell'$, the arc of circle
$$
\{ Q + R e^{is}\ :\ s\in [\beta_\ell', \alpha_{\ell-2}']\},
$$
and $\gamma_{\ell-2}'$. The point $Q$ is the intersection of the segments joining $P_\ell$ with the contact point $P_\ell + e^{i \beta_\ell'}$ and $P_{\ell-2}$ with $P_{\ell-2}+e^{i\alpha_{\ell-2}'}$. Let $M$ denote the midpoint of the segment $P_\ell P_{\ell-2}$. This structure is summarized in Fig. \ref{ballR}.

\begin{figure}[h]                                             
\begin{center}                                                
{\includegraphics[height=5truecm] {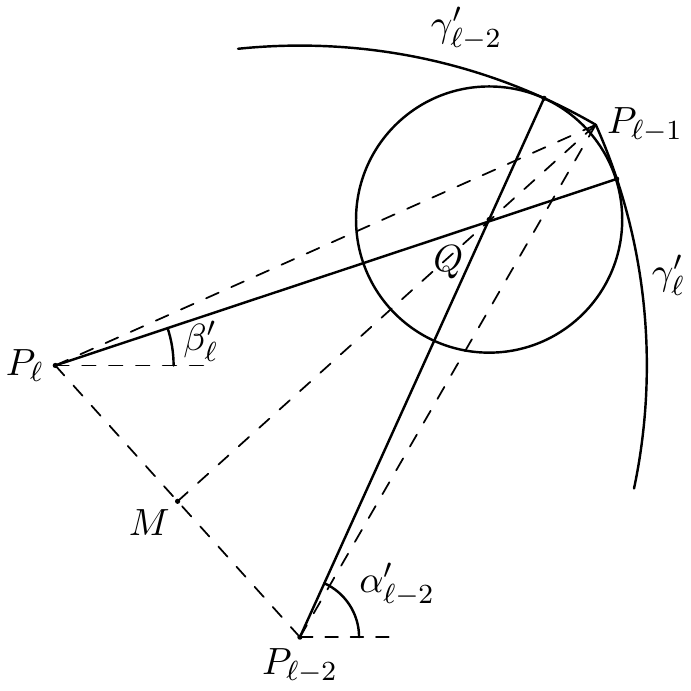}}                                                
\end{center}                                                  
\caption{{\it The boundary of the Cheeger set in a neighborhood of $P_{\ell-1}$.}}\label{ballR} 
\end{figure}                
Note that the angle $\widehat{P_\ell P_{\ell-1} P_{\ell-2}}$ is equal to $\alpha_{\ell-2}-\beta_\ell= j_{\ell-1}$. 
Let us write the point $Q$ in two different ways:
\begin{equation}\label{systemQ}
\left\{
\begin{array}{lll}
Q=P_\ell + (1-R)e^{i \beta_\ell'}
\\
Q= P_\ell + e^{i\beta_\ell} + |QP_{\ell-1}|e^{i (\beta_\ell+\pi + j_{\ell-1}/2)}.
\end{array}
\right.
\end{equation}
In order to compute $|QP_{\ell-1}|$, let us introduce the auxiliary function
\begin{equation}\label{defU}
U(x):= \arcsin (\sin(x)/\sqrt{a}),
\end{equation}
being $a:=(1-R)^2$. According to this notation, we infer that the angle $\widehat {P_\ell Q M}=\widehat{M Q P_{\ell-2}}$ is nothing but $U(j_{\ell-1}/2)$. Therefore, we may write
$$
|QP_{\ell-1}|=|MP_{\ell-1}| - |MQ|= \cos(j_{\ell-1}/2) - \sqrt{a}\cos(U(j_{\ell-1}/2)).
$$
By combining the previous expression with \eqref{systemQ}, we conclude that
\begin{equation}\label{proj1}
\left\{
\begin{array}{lll}
(1-R)\cos(\beta_\ell')&=\cos(\beta_\ell) - \left[\cos(j_{\ell-1}/2) - \sqrt{a}\cos(U(j_{\ell-1}/2))\right]\cos(\beta_\ell+j_{\ell-1}/2)
\\
(1-R)\sin(\beta_\ell')&=\sin(\beta_\ell) -\left[\cos(j_{\ell-1}/2) - \sqrt{a}\cos(U(j_{\ell-1}/2))\right]\sin(\beta_\ell+j_{\ell-1}/2).
\end{array}
\right.
\end{equation}
Similarly, exploiting the fact that
$$
\left\{
\begin{array}{lll}
Q=P_{\ell-2} + (1-R)e^{i \alpha_{\ell-2}'}
\\
Q= P_{\ell-2} + e^{i\alpha_{\ell-2}} + |QP_{\ell-1}|e^{i (\alpha_{\ell-2}+\pi - j_{\ell-1}/2)}
\end{array}
\right.
$$
we get
\begin{equation}\label{proj2}
\left\{
\begin{array}{lll}
(1-R)\cos(\alpha_{\ell-2}')&=\cos(\alpha_{\ell-2}) - \left[\cos(j_{\ell-1}/2) - \sqrt{a}\cos(U(j_{\ell-1}/2))\right]\cos(\alpha_{\ell-2} - j_{\ell-1}/2)
\\
(1-R)\sin(\alpha_{\ell-2}')&=\sin(\alpha_{\ell-2}) -\left[\cos(j_{\ell-1}/2) - \sqrt{a}\cos(U(j_{\ell-1}/2))\right]\sin(\alpha_{\ell-2} - j_{\ell-1}/2).
\end{array}
\right.
\end{equation}

\subsection{Blaschke deformations}
We now introduce a family of deformations in the class of Reuleaux polygons of width 1, which allow to connect any pair of elements in a continuous way (with respect to the complementary Hausdorff distance), staying in the class. This definition has been introduced by W. Blaschke in \cite{Bla} and analysed by Kupitz-Martini in \cite{KM}.

\begin{definition}\label{def-BD}
Let $\Omega$ be a Reuleaux polygon with $2N+1$ sides. Let $k$ be one of the indexes in $\{1,\ldots, 2N+1\}$. A \emph{Blaschke deformation} acts moving the point $P_{k}$ on the arc $\gamma_{k-1}$ increasing or decreasing the arc length. Consequently, the point $\gamma_{k+1}$ moves and the arcs $\gamma_{k}$, $\gamma_{k+1}$, and $\gamma_{k+2}$ are deformed, as in Fig. \ref{fig-def}. We say that a Blaschke deformation is \emph{small} if the arc length of $\gamma_{k-1}$ has changed of $\e\in \mathbb R$, small in modulus. 
\end{definition}

\begin{figure}[h]                                             
\begin{center}                                                
{\includegraphics[height=4.5truecm] {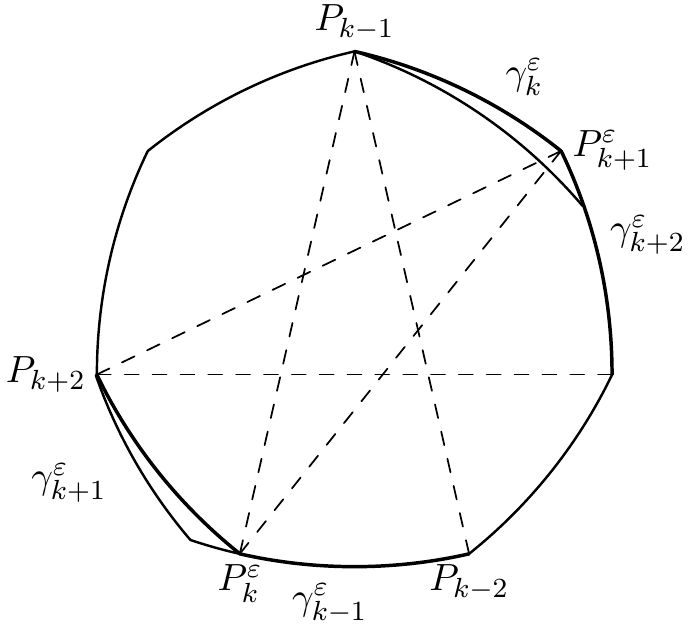}}                                                
\end{center}                                                  
\caption{{\it A Blaschke deformation of a Reuleaux heptagon which moves $P_k$ on $\gamma_{k-1}$ changing $\alpha_{k-1}$ into $\alpha_{k-1}^\e:=\alpha_{k-1}+\e$, with $\e>0$ small.}}\label{fig-def} 
\end{figure}

Let $\Omega_\e$ denote the Reuleaux polygon obtained by $\Omega$ after a small Blaschke deformation of parameter $\e$, moving $P_k$. When $\e$ is infinitesimal, $\Omega_\e$ can be written as the image of a small perturbation of the identity: 
$$
\Omega_\e = \phi_\e (\Omega),\quad \phi_\e (x)= x + \e V(x) + o(|\e|),
$$
for a suitable vector field $V$. The behavior of $V$ on the boundary is described in  \cite[\S 2.1, formulas (2.5) and (2.6)]{HL}, in particular, adopting the parametrization \eqref{param} of the boundary arcs, there holds:
\begin{equation}\label{Vn}
V\cdot n = \left\{
\begin{array}{lll}
\sin(s-\alpha_{k-1}) \quad & \hbox{on }\gamma_k
\\
\displaystyle{-\frac{\sin j_k}{\sin j_{k+1}}\sin(s-\alpha_{k+1})}\quad & \hbox{on }\gamma_{k+1}
\\
0 & \hbox{else},
\end{array}
\right.
\end{equation}
where, according to \eqref{defj}, $j_k$ and $j_{k+1}$ are the lengths of $\gamma_k$ and $\gamma_{k+1}$, respectively.

\subsection{The first order shape derivative of $h$ with respect to Blaschke deformations}
In order to derive optimality conditions, a classical idea is to impose that the first order shape derivative of $h$ at a critical Reuleaux polygon vanishes for every small deformation which preserves the constraints of $\mathcal B^1_N$.

For a generic convex set $\Omega$, denoting by $C_\Omega$ its (unique, see \cite{AlCa}) Cheeger set, the first order shape derivative of $h$ at $\Omega$ in direction $V\in C^1(\mathbb R^2;\mathbb R^2)$ reads (see \cite{PaSa})
\begin{equation}\label{hprime}
\frac{\de h}{\de V}(\Omega) :=\lim_{\varepsilon \to 0}\frac{h((I + \varepsilon V)(\Omega) )- h(\Omega)}{\varepsilon}=\frac{1}{|C_\Omega|} \int_{\partial C_\Omega \cap \partial \Omega} (\mathcal C_\Omega - h(\Omega)) V\cdot n\, \de \mathcal H^{1},
\end{equation}
where $I$ is the identity map and  $\mathcal C_\Omega$ is the curvature.

A consequence of this formula is that the Cheeger set of a maximizer in $\mathcal B_N^1$ intersects all the boundary arcs.
\begin{proposition}\label{boundary}
Let $\Omega^*$ be a maximizer for $h$ in $\mathcal B_N^1$. Then $\partial C_\Omega \cap \gamma_\ell \neq \emptyset$ for every $\ell$.
\end{proposition}
\begin{proof}
Let us consider an infinitesimal Blaschke deformation which moves the vertex $P_k$ (see Definition \ref{def-BD}). Since $\Omega$ is a critical shape, the first order shape derivative of $h$ at $\Omega$ with respect to this deformation is zero. In formulas, exploiting \eqref{hprime}, \eqref{Vn}, and the fact that the curvature $\mathcal C_\Omega$ is equal to 1 on $\partial \Omega \cap \partial C_\Omega$, we derive the following optimality condition:  
\begin{equation}\label{op1}
0 =  \int_{\partial C_\Omega \cap \gamma_k } \sin(s-\alpha_{k-1})\,\de s -\frac{\sin j_k}{\sin j_{k+1}}
\int_{\partial C_\Omega \cap \gamma_{k+1}} \sin (s-\alpha_{k+1})\, \de s.
\end{equation}
Looking at \eqref{op1}, we see that if the Cheeger set does not meet the arc $\gamma_k$,
the corresponding integral over $\partial C_\Omega \cap \gamma_k$ is zero, therefore the
other integral has to be zero, meaning that either the Cheeger set does not meet $\gamma_{k+1}$ or the intersection is a single point.
Repeating this argument, we obtain that $\partial C_\Omega \cap \gamma_\ell$ is either empty
or a singleton, for any arc $\gamma_\ell$ on the boundary.

Let us prove that it is impossible (this is a common fact for any Cheeger set).
Since the free parts of $C_\Omega$ are arcs of circle of radius $R(\Omega)$, if the Cheeger
set touches the boundary of $\Omega$ only at singletons, we would have that the curvature $\mathcal{C}_{C_\Omega}$ of
the Cheeger set is equal to $1/R(\Omega)$ almost everywhere. Thus by Gauss-Bonnet formula:
$$\mathcal{H}^1(\partial C_\Omega)=\int_{\partial C_\Omega} \mathcal{C}_{C_\Omega} X\cdot n=\frac{1}{R(\Omega)} \int_{\partial C_\Omega} X\cdot n = \frac{2|C_\Omega|}{R(\Omega)},$$
which would yield 
$$h(\Omega)= \frac{\mathcal{H}^1(\partial C_\Omega)}{|C_\Omega|}=\frac{2}{R(\Omega)},$$
in contradiction with \eqref{defR}.
\end{proof}

Taking $\Omega$ a Reuleaux polygon and $V$ inducing an arbitrary Blaschke deformation, we obtain a family of optimality conditions. In order to state the result, let us introduce the following auxiliary functions:
\begin{eqnarray}
&& G(x):=\sin^2 (x) + \sqrt{a} \cos(x) \cos(U(x))
=\sin^2 (x) + \cos(x) \sqrt{a-\sin^2(x)},\label{defg}
\\
&& F(x,y):=\sqrt{a}\cos(2x+y-U(y)), \label{defF}
\\
&& H(x,y,z):=\sin(2z)[g(x)-F(y,z)],\label{defH}
\end{eqnarray}
where $a$ is a constant depending on $\Omega$, $a:=(1-R(\Omega))^2$, and $U$ is the function defined in \eqref{defU}. 
\begin{proposition}\label{propoc} Let $\Omega$ be a critical shape for $h$ in the class of Reuleaux polygons. Then, for every $k$,
\begin{equation}\label{oc3}
H\left( \frac{j_{k-1}}{2}, \frac{j_{k}}{2}, \frac{j_{k+1}}{2}   \right)=
H\left( \frac{j_{k+2}}{2}, \frac{j_{k+1}}{2}, \frac{j_{k}}{2}   \right).
\end{equation}
\end{proposition}
\begin{proof} In the following, for brevity, the functional $R(\Omega)$ will be denoted by $R$.

We use again Formula  \eqref{op1}.
In view of the parametrization \eqref{arcprime} of $\partial C_\Omega\cap \partial \Omega$, we get
\begin{align*}
0 & = \int_{\alpha_k' }^{\beta_k'} \sin(s-\alpha_{k-1})\,\de s -\frac{\sin j_k}{\sin j_{k+1}}
\int_{\alpha_{k+1}'}^{\beta_{k+1}'} \sin (s-\alpha_{k+1})\, \de s
\\
& = -\cos(\beta_k'-\alpha_{k-1}) + \cos(\alpha_k'-\alpha_{k-1}) + \frac{\sin j_k}{\sin j_{k+1}} \left[\cos(\beta_{k+1}'-\alpha_{k+1}) - \cos (\alpha_{k+1}' -\alpha_{k+1}) \right]
\\
& =  \cos(\beta_k'-\beta_k) - \cos(\alpha_k'-\beta_k) + \frac{\sin j_k}{\sin j_{k+1}} \left[\cos(\beta_{k+1}'-\alpha_{k+1}) - \cos (\alpha_{k+1}' -\alpha_{k+1}) \right],
\end{align*}
where in the last equality we have used $\alpha_{k-1}\equiv\beta_k + \pi$, modulo $2\pi$. 
Rearranging the terms, we rewrite the optimality condition as:
\begin{equation}\label{oc1}
\sin (j_{k+1}) [\cos (\beta_k-\beta'_k) - \cos(\beta_k-\alpha'_k)] = \sin(j_k) [\cos (\alpha'_{k+1}-\alpha_{k+1}) - \cos(\beta'_{k+1}-\alpha_{k+1})].
\end{equation}
Exploiting \eqref{proj1} with $\ell=k$, we obtain
\begin{align}
\cos(\beta_k-\beta_k') &= \cos(\beta_k)\cos(\beta_k')+\sin(\beta_k)\sin(\beta_k')\notag
\\ & = \left[\sin^2(j_{k-1}/2) +\sqrt{a} \cos(j_{k-1}/2) \cos(U(j_{k-1}/2))\right]/\sqrt{a}\notag
\\
& =G\left(\frac{j_{k-1}}{2}\right)/\sqrt{a},\label{term1}
\end{align}
where $G$ is the function defined in \eqref{defg} and $a:=(1-R)^2$.
Similarly, taking $\ell=k+2$ in \eqref{proj2}, we get
\begin{align}
\cos(\beta_k-\alpha_k') &=  \left[ \cos(j_k) - \left( \cos(j_{k+1}/2)-\sqrt{a} \cos(U(j_{k+1}/2))\right)\cos(j_k+j_{k+1}/2) \right]/\sqrt{a}\notag
\\
&= \left[\sin(j_k)\sin (j_k + j_{k+1}/2) + \sqrt{a} \cos(U(j_{k+1}/2))\cos(j_k+j_{k+1}/2) \right] /\sqrt{a}\notag
\\
& =F\left(\frac{j_k}{2}, \frac{j_{k+1}}{2}\right)/\sqrt{a},\label{term2}
\end{align}
where $F$ is the function introduced in \eqref{defF}.
Here, for the last equality, we have used $\cos(x) \cos(y)=\cos(x-y)-\sin(x)\sin(y)$, together with the fact $\sqrt{a}\sin (U(x))=\sin(x)$.
By combining \eqref{term1} with \eqref{term2}, we may rewrite the left-hand side of \eqref{oc1} as
$$
H\left( \frac{j_{k-1}}{2}, \frac{j_{k}}{2}, \frac{j_{k+1}}{2}   \right)/\sqrt{a}.
$$ 
The same strategy adopted for the left-hand side of \eqref{oc1} also applies for the right-hand side, giving $G\left(\frac{j_{k+2}}{2}\right)/\sqrt{a}-F\left(\frac{j_{k+1}}{2}, \frac{j_{k}}{2}\right)/\sqrt{a} = H\left( \frac{j_{k+2}}{2}, \frac{j_{k+1}}{2}, \frac{j_{k}}{2}   \right)/\sqrt{a}$.
By multiplying both sides by $\sqrt{a}$, we get \eqref{oc3}.
\end{proof}
\begin{remark}
The optimality condition \eqref{oc3} is obviously satisfied by all the regular Reuleaux
polygons. Actually, we believe that only the regular Reuleaux polygons are critical points
for the Cheeger constant, and we give some support to this claim in Remark \ref{rem3.6}.
If we were able to prove that fact, the proof of our main theorem would be much simpler
as we could make the explicit computation of the Cheeger constant of any regular
Reuleaux polygon. Let us also refer to the recent paper \cite{Phi} where is done
a similar comparison between the first Dirichlet eigenvalue or the torsion among
every regular Reuleaux polygon, showing that the Reuleaux triangle is always the
optimal domain in this restricted class.
\end{remark}
\subsection{Analysis of the optimality conditions}
Throughout the subsection $\Omega$ will denote a maximizer for $h$ in $\mathcal B^1_M$ for some $M$, with $2N+1$ boundary arcs $j_1, \ldots, j_{2N+1}$.
We will use the optimality conditions stated in Propositions \ref{boundary} and \ref{propoc} to obtain some information on the lengths of consecutive intervals. 

\medskip

A first rough estimate can be deduced from Propositions \ref{boundary}:
\begin{lemma}\label{roughesti}
Two consecutive lengths $j_k$ and $j_{k+1}$ satisfy
$$
0.1339 j_{k+1} \leq j_k \leq \frac{1}{0.1339} j_{k+1}.
$$
\end{lemma}
\begin{proof}
In subsection \ref{ssR}, we have highlighted the relation between the lengths of the arcs $\gamma_\ell'\subset \partial \Omega \cap \partial C_\Omega$ and the arcs $\gamma_\ell$. This relation, valid for every $\ell$ thanks to Proposition \ref{boundary}, can be stated as follows (see also Fig. \ref{ballR}  with $\ell=k+2$):
$$\alpha'_k-\alpha_k = \arcsin\left(\frac{\sin j_{k+1}/2}{1-R}\right)-\frac{j_{k+1}}{2}=U\left(\frac{j_{k+1}}{2}\right) - \frac{j_{k+1}}{2}.$$
Now $j_k\geq \alpha'_k-\alpha_k$, thus using standard estimates for the arcsine and sine, we obtain
$$
j_k \geq \arcsin\left(\frac{1}{1-R}\sin\left(\frac{j_{k+1}}{2}\right)\right) - \frac{j_{k+1}}{2}  \geq  \frac{R}{2(1-R)} j_{k+1} \geq 0.1339 j_{k+1}.
$$
For the last inequality, we have used $R\geq 0.21132$, see \eqref{estimatesR}.
The other bound in the statement can be obtained in a similar way, by considering the difference of $\beta_{k+1}-\beta'_{k+1}$.
\end{proof}

The optimality condition \eqref{oc3} in Proposition \ref{propoc} allows us to obtain a refined estimate.
\begin{theorem}\label{theolengths}
Two consecutive lengths $j_k$ and $j_{k+1}$ satisfy
\begin{equation}
\tau j_{k+1} \leq j_k \leq \frac{1}{\tau} j_{k+1}\ \mbox{with $\tau \geq 0.99 - 0.05 h^2$}
\end{equation}
where $h$ denotes the largest length (among all arcs of the Reuleaux polygon).
\end{theorem}
Before showing the proof of this result, let us make some comments about its consequences.

\begin{remark}\label{rem3.6} Since the maximal length of any arc is less than $\pi/3$ (a maximal arc joins two points on the outercircle
being tangent to the incircle: this computation is done for example in \cite{HL}), we deduce from the
theorem that $\tau\geq 0.93$ and the smaller is $h$, the better will be the estimate.
For example, for $h\leq 0.45$, we get $\tau \geq 0.97$. In some sense, the optimal domain
is close to a regular Reuleaux polygon. Note that if the polygon has $2N+1$ sides, the smallest one
has a length that is at least $\tau^N h$ (because there is at most $N-1$ arcs between the largest and 
the smallest if we turn in the good direction).
\end{remark}

We can deduce from Theorem \ref{theolengths} a bound for the maximal length of an arc
of a $2N+1$ optimal Reuleaux polygon.
\begin{proposition}\label{propmaxlength}
Let $h_N$ be the maximal length of an arc of a $2N+1$ Reuleaux polygon satisfying the optimality
conditions. Let us denote by $\tau_N$ the rate between two consecutive lengths as defined in Theorem \ref{theolengths}. Then
\begin{equation}\label{maxlength}
h_N\leq \frac{(1-\tau_N)\, \pi}{1+\tau_N - 2\tau_N^{N+1}}.
\end{equation}
\end{proposition}
\begin{remark}
Thanks to this proposition, we obtain for example the following bounds for a $2N+1$ Reuleaux polygon
satisfying the optimality conditions: using the fact that the right-hand side of \eqref{maxlength} is decreasing in $\tau$ and iterating Theorem \ref{theolengths} to get better rates we have, we infer that the maximal length of one side $h_N^{max}$ and the minimal length of one side $h_N^{min}$ (that is computed as $\tau_N^N h_N^{max}$) satisfy: \\
\begin{center}
\begin{tabular}{c|c|c|c|c}
$N$ & $2N+1$ & $\tau_N$ & $h_N^{max}$ & $h_N^{min}$ \\ \hline
2 & 5 &  0.9687 & 0.6526 & 0.6123 \\
3 & 7 & 0.9791 &  0.4652 & 0.4367 \\
4 & 9 & 0.9834 & 0.3622   & 0.3387\\
5 & 11 & 0.9855 & 0.2971 & 0.2762 \\
6 & 13 & 0.9868 &  0.2522  & 0.2328 \\
7 & 15 & 0.9875 & 0.2194 & 0.2009 \\
8 & 17 & 0.9881 & 0.1944 & 0.1765 \\
9 & 19 & 0.9884 & 0.1746 & 0.1572 \\
\end{tabular}
   \captionof{table}{Table of rates, maximal and minimal lengths for Reuleaux polygons}\label{table1}
\end{center}

confirming that we are not far from regular Reuleaux polygons.
\end{remark}
\begin{proof}[Proof of Proposition \ref{propmaxlength}]
We start from the arc of length $h_N$. Its two neighbours have a length at least $\tau_N h_N$, the next
neighbours have a length at least $\tau_N^2 h_N$... up to the farthest arcs (in the enumeration)
 which have a length at least $\tau_N^N h_N$. Therefore, since the sum of all the lengths is equal to
 the perimeter $\pi$, we have the inequality
$$\pi \geq h_N\left(1+2\sum_{k=1}^N \tau_N^k\right)=h_N \frac{1+\tau_N -2\tau_N^{N+1}}{1-\tau_N}$$
therefore
\begin{equation}\label{majhN}
h_N\leq \frac{\pi (1-\tau_N)}{1+\tau_N -2\tau_N^{N+1}}.
\end{equation}
\end{proof}

The remaining part of the subsection is devoted to the proof of Theorem \ref{theolengths}. We start with some inequalities for the functions $U,G,F,H$ defined above, see \eqref{defU} and \eqref{defg}-\eqref{defH}. 
We assume that our variables satisfy the following conditions:
$$
x,y,z \in \left [0,\frac{\pi}{6}\right],\quad 2y\geq U(z)-z.
$$
We will also use everywhere, for the higher order terms the estimate \eqref{estimatesR}, namely $R\in [0.21132,0.22803]$.
Starting from
$$x-\frac{x^3}{6}\leq \sin x\leq x - 0.16439 x^3\ \;\mbox{for } 0\leq x\leq \frac{\pi}{6}$$

we get successively
\begin{equation}\label{estiu}
\frac{x}{1-R} + 0.1284 x^3 \leq U(x) \leq \frac{x}{1-R} + 0.1831 x^3
\end{equation}

$$
1-R - \frac{R^2x^2}{2(1-R)} - 0.05 x^4 \leq G(x) \leq 1-R - \frac{R^2x^2}{2(1-R)} - 0.02 x^4
$$
$$F(y,z)\leq 1-R-2(1-R)y^2-\frac{R^2 z^2}{2(1-R)}+2Ryz+S_1(y,z)$$
with a remainder
$$S_1(y,z)=  0.52579 y^4- 0.28176 y^3z +  0.06736 y^2z^2 +0.28376 yz^3- 0.03852 z^4.$$
In the same way
$$F(y,z)\geq 1-R-2(1-R)y^2-\frac{R^2 z^2}{2(1-R)}+2Ryz+S_2(y,z)$$
with a remainder
$$S_2(y,z)= 0.49223 y^4-0.30404 y^3z + 0.0403 y^2z^2 +0.19159  yz^3-0.02727 z^4.$$
Putting together these different estimates yields for $H(x,y,z)$ the following estimates:
\begin{equation}\label{estiH1}
H(x,y,z)\leq 4(1-R)y^2z+\frac{R^2}{1-R}(z^3-zx^2)-4Ryz^2+T_1(x,y,z)
\end{equation}
with a remainder which is a polynomial of degree 5 given by
\begin{align*}T_1(x,y,z)= & -0.04 x^4z +0.00732 x^3z^2 + 0.04491 x^2z^3 -0.98447 y^4z 
\\ & +0.60808 y^3z^2 -1.94515 y^2z^3 +0.26158 yz^4 +0.01257 z^5,
\end{align*}
and
\begin{equation}\label{estiH2}
H(x,y,z)\geq 4(1-R)y^2z+\frac{R^2}{1-R}(z^3-zx^2)-4Ryz^2+T_2(x,y,z)
\end{equation}
with a remainder given by
\begin{align*}
T_2(x,y,z)=& -0.1x^4z +  0.03774 x^2z^3 - 1.05158 y^4z 
\\ & +0.56352 y^3z^2 -2.21639 y^2z^3 -0.004 yz^4 + 0.02293 z^5.
\end{align*}
Now writing the optimality condition $H\left( \frac{j_{k-1}}{2}, \frac{j_{k}}{2}, \frac{j_{k+1}}{2}   \right)-
H\left( \frac{j_{k+2}}{2}, \frac{j_{k+1}}{2}, \frac{j_{k}}{2}   \right)=0$ and using estimates
\eqref{estiH1}, \eqref{estiH2} yields the two following inequalities
\begin{equation}\label{estiH3}
0\leq j_k j_{k+1}(j_k-j_{k+1}) + \frac{R^2}{4(1-R)}\left(j_{k+1}^3-j_{k+1}j_{k-1}^2+j_kj_{k+2}^2-j_k^3\right) +E_1(j_{k-1},j_k,j_{k+1},j_{k+2})/16
\end{equation}
with
\begin{eqnarray*}
E_1(j_{k-1},j_k,j_{k+1},j_{k+2})=-0.04 j_{k-1}^4j_{k+1}+0.00732 j_{k-1}^3j_{k+1}^2+0.04491 j_{k-1}^2j_{k+1}^3 \\
-0.02293 j_k^5 
-0.98047 j_k^4j_{k+1} +2.82447 j_k^3j_{k+1}^2-0.03774 j_k^3j_{k+2}^2
\\
-2.50867j_k^2j_{k+1}^3 +1.31316 j_kj_{k+1}^4 +0.1 j_kj_{k+2}^4
+0.01257 j_{k+1}^5
\end{eqnarray*}
and
\begin{equation}\label{estiH4}
0\geq j_k j_{k+1}(j_k-j_{k+1}) + \frac{R^2}{4(1-R)}\left(j_{k+1}^3-j_{k+1}j_{k-1}^2+j_kj_{k+2}^2-j_k^3\right) +E_2(j_{k-1},j_k,j_{k+1},j_{k+2})/16
\end{equation}
with
\begin{eqnarray*}
E_2(j_{k-1},j_k,j_{k+1},j_{k+2})=-0.1 j_{k-1}^4j_{k+1}+0.03774 j_{k-1}^2j_{k+1}^3
-0.01257 j_k^5\\ -1.31316 j_k^4j_{k+1} 
+ 2.50867 j_k^3j_{k+1}^2
-0.04491 j_k^3j_{k+2}^2-2.82447 j_k^2j_{k+1}^3
\\-0.00732 j_k^2j_{k+2}^3+0.98047 j_kj_{k+1}^4 +0.04 j_kj_{k+2}^4 + 0.02293 j_{k+1}^5.
\end{eqnarray*}
Note that the coefficient $R^2/(4(1-R))$ satisfies
\begin{equation}\label{016}
0.01415 \leq \frac{R^2}{4(1-R)} \leq 0.01684.
\end{equation}

We are going to make the proof of Theorem \ref{theolengths} in the case $j_k\leq j_{k+1}$ by using inequality \eqref{estiH3} which leads to a simple inequality for a polynomial of degree 3 easy to analyse. In the case
$j_k\geq j_{k+1}$ we can proceed exactly in the same way, but using inequality \eqref{estiH4} instead.

In the sequel we use the following notations: $j_{k+1}=j$ and $j_{k-1}=u j$, $j_k=t j$, $j_{k+2}=v j$.
The three numbers $t,u,v$ are positive and $t\leq 1$ by assumption.
First, we need an estimate of the remainder $E_1$, that can be written as
$$
E_1(u j , tj, j, vj)
=j^5\left(F_1(t,v) + F_2(u)\right)$$
with
$$
F_1(t,v)=   -0.02293  t^5 -0.98047 t^4 +  2.82447  t^3-0.03774 t^3 v^2 
 -2.50867 t^2 +1.31316 t + 0.1 t  v^4 + 0.01257
$$
and
$$
 F_2(u)=   -0.04 u ^4  +  0.00732 u^3+ 0.04491  u^2.
 $$
\begin{lemma}\label{estiE1}
For $v\leq 1$, we have $E_1(u j , tj, j, vj)\leq 0.71653873 j^5$ (for every $t\in [0,1]$ and every $u$).
\\
For $v\geq 1$, we have $E_1(u j , tj, j, vj)\leq (0.1v^4- 0.03774 v^2+ 0.65427873) j^5$ (for every $t\in [0,1]$ and every $u$). 
\end{lemma}
\begin{proof}
The expression containing $u$ being a polynomial of degree 4, it is straightforward to prove that the
maximum of $F_2$ on $\mathbb{R}_+$ is attained for $u\in [0.8210107,0.8210108]$  
and its value is less than $0.01614873$.

For $F_1(t,v)$, computing the derivative with respect to $v$, we see that this derivative is negative
when $v\leq \sqrt{0.07548/0.4}\, t$ 
and positive after. Since it is easy to check that, for all $t\in [0,1]$,
$F(t,0)< F(t,1)$ we see that, to maximize $F_1$, we must choose $v=1$ when $v\leq 1$, or keep
$v$ when $v\geq 1$. Now, the derivative of $F_1$ with respect to $t$ is
$$\frac{\partial F_1}{\partial t}=- 0.11465 t^4 -3.92188  t^3 +  8.4734  t^2 
 -5.01734 t +1.31316 - 0.11322 t^2 v^2 + 0.1 v^4.$$
The two last terms satisfy
$$0.1 v^4 - 0.11322 t^2 v^2  \geq 0.1 v^4 - 0.11322 v^2 \geq -0.0033,$$
thus we are led to study a polynomial of degree 4 in $t$ and it is immediate to check that this polynomial
is positive for any $t\in [0,1]$. Therefore, to maximize $F_1$ we must choose $t=1$. The conclusion follows.
\end{proof}

\medskip
We are now in a position to prove the theorem.
\begin{proof}[Proof of Theorem \ref{theolengths}]
We make the proof in two steps. In the first step, we prove that $\tau \geq 0.92$ 
and then, using this estimate we get the conclusion.\\
Assuming $j_k\leq j_{k+1}$, we have to consider four different cases:
\begin{itemize}
\item Case 1: $j_{k-1}\geq j_k$ and $j_{k+1}\geq j_{k+2}$
\item Case 2: $j_{k-1}\leq j_k$ and $j_{k+1}\geq j_{k+2}$
\item Case 3: $j_{k-1}\leq j_k$ and $j_{k+1}\leq j_{k+2}$
\item Case 4: $j_{k-1}\geq j_k$ and $j_{k+1}\leq j_{k+2}$
\end{itemize}
We start from the inequality established in \eqref{estiH3}: recalling the upper bound in \eqref{016} for $R^2/(4(1-R))$, we obtain
\begin{equation}\label{estiH3bis}
0\leq t(t-1)+0.01684(1-u^2+t v^2-t^3) + \frac{E_1(uj,tj,j,vj)}{16j^3}.
\end{equation}
The idea is to bound it from above in each of the four cases with a polynomial $q(t)$ of degree 3 in $t$ which satisfies the following properties:
\begin{equation}\label{condiq}
q(0)>0,\quad q(1)>0, \quad q\ \hbox{decreasing and then increasing in }[0,1]\,,\quad q<0 \ \hbox{in }[0.1,0.92].
\end{equation} 
The positivity of the polynomial, together with the estimate $t>0.1339$ provided in Lemma \ref{roughesti}, give $t>0.92$. In each case, we will have to consider a polynomial $q$,
depending on three coefficients $A,B,C$ given
by $q(t)= t(t-1)+0.01684(1+At-Bt^2-t^3) +C$.
It is immediate to check that if $A,B,C$ satisfy
\begin{equation}\label{condiqs}
0.015 \leq B\leq A\leq 1.2\ \mbox{and } \; 0<C\leq 0.0514
\end{equation}
then \eqref{condiq} holds true.

\smallskip

\noindent {\bf Case 1:} Here $u\geq t$ and $v\leq 1$, thus \eqref{estiH3bis}, Lemma \ref{estiE1}, and the fact that $j\leq \pi/3$, give
$$0\leq t(t-1) +0.01684(1-t^2+t-t^3)+0.04912.$$
This polynomial satisfies \eqref{condiqs} and then \eqref{condiq}.

\smallskip

\noindent {\bf Case 2:} Here by assumption $u\leq t$ and $v\leq 1$. Moreover, using Lemma \ref{roughesti}, we get $u\geq 0.1339 t$. These estimates, together with Lemma \ref{estiE1}, imply that
$$
0\leq t(t-1) +0.01684(1-0.1339^2 t^2+t-t^3)+0.04912.
$$ 
This polynomial satisfies \eqref{condiqs} and then \eqref{condiq}.

\smallskip

\noindent 
{\bf Case 3:} In this case, the four lengths are increasing and may belong to a sequence of
increasing numbers $j_{k-1} \leq j_k \leq j_{k+1} \leq j_{k+2} \ldots \leq j_m$ with $j_m\geq j_{m+1}$.
We proceed by a descent induction. From the case 2., we see that $j_{m-1}\geq 0.92 j_m$. 
Assume by induction that $j_{k+1}\geq 0.92 j_{k+2}$. 
Then $v\leq 1/0.92$.
To estimate $j_k$ in terms of $j_{k+1}$ we use the
same technique as above with 
$$
1-u^2+t v^2-t^3\leq 1-0.1339^2 t^2 + t/0.92^2 -t^3.
$$
We have also to change
the estimate for the remainder $E_1$: according to Lemma \ref{estiE1} and since $1\leq v\leq1/0.92$ 
we have here 
$$
\frac{E_1(uj,tj,j,vj)}{16 j^3} \leq (0.1/0.92^4-0.03774/0.92^2+0.65427873) \frac{j^2}{16}\leq 0.051355.
$$
All in all, \eqref{estiH3bis} gives
$$
0 \leq t(t-1) + 0.01684 (1-0.1339^2 t^2 + t/0.92^2 -t^3) + 0.051355.
$$
This polynomial satisfies \eqref{condiqs} and then \eqref{condiq}.

\smallskip

\noindent 
{\bf Case 4:} We proceed as in the third case, by induction starting at the last number $j_m$ of the
increasing sequence $j_k\leq  j_{k+1} \leq j_{k+2} \ldots \leq j_m$. Here the upper bound of \eqref{estiH3bis} is 
$$
t(t-1) + 0.01684(1-t^2+t/0.92^2-t^3)+ 0.051355.
$$
and satisfies \eqref{condiqs} and \eqref{condiq}.

\medskip
We now get a better estimate by using this number $0.92$, without replacing $j$ by its upper bound $\pi/3$. 
Since the statement is valid for any pair of consecutive lengths, coming back to $t,u,v$, we have: 
$t\leq 1, u\geq 0.92 t,0.92 \leq v\leq 1/0.92$. Using these stronger estimates in \eqref{estiH3bis}, we get
$$
0\leq q(t):=t(t-1) + 0.01684(1-0.92^2 t^2+t/0.92^2-t^3) + 0.04682 j^2.
$$
Here the constant term is the same as in Cases 3 and 4. The polynomial satisfies \eqref{condiq}. Therefore, if we show that its larger root is at least $0.99-0.05j^2$ we are done. In other words, we shall prove that $q(0.99-0.05 j^2)\leq 0$. Developing the computation, we obtain
$$
q(0.99-0.05j^2)\leq 10^{-3}(-3.672 + 0.722 j^2 + 2.34 j^4 + 0.002 j^6)
$$
and since the right--hand side is negative for $j\leq \pi/3$ the thesis follows.
\end{proof}

\subsection{Inradius of a Reuleaux polygon}\label{secinradius}
In this paragraph we give a general formula for the inradius of a Reuleaux polygon.

\begin{definition}\label{def-contact}
We say that $M\in \partial \Omega$ is a {\it contact point} if it belongs to the incircle.
\end{definition}

In the particular case in which there exist three contact points, two belonging to consecutive arcs, and the third belonging to the opposite arc, the inradius can be easily computed. 
\begin{lemma}\label{r5}
Let $\Omega$ be a Reuleaux polygon with $2N+1$ sides. Assume that the arcs $\gamma_1$, $\gamma_2$, and $\gamma_{2N+1}$ are tangent to the incircle. Then 
\begin{equation}\label{formula5}
r(\Omega)=1-\frac{1}{2\cos(j_1/2)}.
\end{equation}
\end{lemma}
\begin{proof} The standing assumptions are summarized in Fig. \ref{fig-3}. 

\begin{figure}[h]                                             
\begin{center}                                                
{\includegraphics[height=3.5truecm] {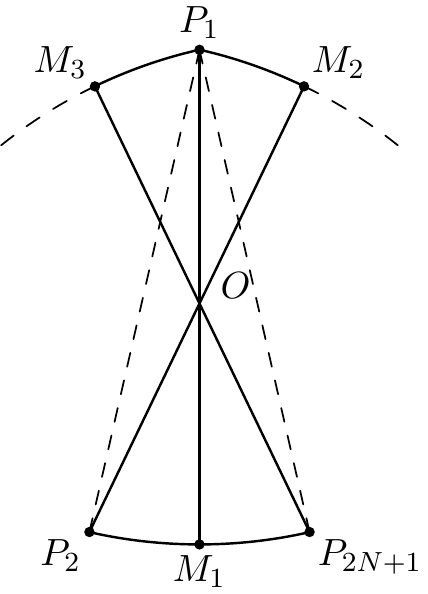}}                                                
\caption{{\it The configuration under study: contact points on two consecutive arcs $\gamma_2$, $\gamma_{2N+1}$, and on the opposite arc $\gamma_1$.}}\label{fig-3}
\end{center}                                                  
\end{figure}

Since the arcs $\gamma_1$,$\gamma_2$, and $\gamma_{2N+1}$ are tangent to the incircle, denoting by $M_1$, $M_2$, and $M_3$ the three contact points, respectively, we infer that the length of $OM_1$, $OM_2$, and $OM_3$ are equal to $r(\Omega)$, so that the length of $OP_1$, $OP_2$, and $OP_{2N+1}$ are equal to $1-r(\Omega)$. In particular, the two triangles $P_1OP_2$ and $P_1OP_{2N+1}$ are congruent. Let us consider one of the two triangles: it is isosceles, with base of length 1, legs of length $1-r(\Omega)$, and base angle of amplitude $j_1/2$. Therefore, we conclude that
$$
\big(1-r(\Omega)\big) \cos(j_1/2)= \frac12.
$$
This concludes the proof.
\end{proof}

\begin{remark}\label{penta}
The described situation in Lemma \ref{r5} always occurs for regular Reuleaux polygons (actually, in this case all the boundary arcs are tangent to the incircle) and for all the Reuleaux pentagons.
\end{remark}

The previous situation is a particular case: what remains true in general is the existence of three contact points which do not lie in the same half-plane (limited by a line going through the origin); what changes is the number of boundary points between pairs of contact points. In order to clarify this fact, we need to introduce the notion of {\it sector}.

\begin{definition}\label{def-sec} 
Let $M_1,M_2,M_3$ be three contact points (labeled in the direct sense) with polar angles $t_1,t_2,t_3$, and not lying in the same half-plane (limited by a line going through the origin). The segments joining these contact points with their opposite boundary points (with polar angles $t_i+\pi$) pass through the origin and identify a partition of the interval $[0,2\pi]$ in six parts, that we call {\it sectors}:
$$[t_2,t_1+\pi],[t_1+\pi,t_3],[t_3,t_2+\pi],[t_2+\pi,t_1],[t_1,t_3+\pi],[t_3+\pi,t_2].$$
Here the angles are intended modulo $2\pi$. The length of the sector $[t_i, t_{i-1}+\pi]$, $i\in \mathbb Z_3$, is denoted by $u_i$ or, when no ambiguity may arise, simply by $u$. 
\end{definition}

\begin{remark}
Note that each sector  $[t_i, t_{i-1}+\pi]$, $i\in \mathbb Z_3$, is coupled with the opposite sector $[t_i+\pi, t_{i-1}]$ (again, angles intended modulo $2\pi$) which has the same length.
\end{remark}

Beside the length, we associate to each sector another characteristic parameter. To fix the ideas, let us consider the first sector $[t_2,t_1+\pi]$. Up to relabeling the indexes, we may assume that $M_1$ belongs to the boundary arc $\gamma_1$ centered at $P_1$. Accordingly, $M_2$ lies on the boundary arc $\gamma_{2m}$ centered at $P_{2m}$, for some $m$. Going along the boundary in the direct sense, namely in counter-clockwise sense, between $M_2$ and $P_1$, we find $P_{2m-1}, P_{2m-3}, \ldots, P_3, P_1$; whereas between $P_{2m}$ and $M_1$ we find the vertexes $P_{2m}, P_{2m-2}, \ldots, P_4, P_2$ (see also Fig. \ref{fig-t}).
\begin{figure}[h]                                             
\begin{center}                                                
{\includegraphics[height=4truecm] {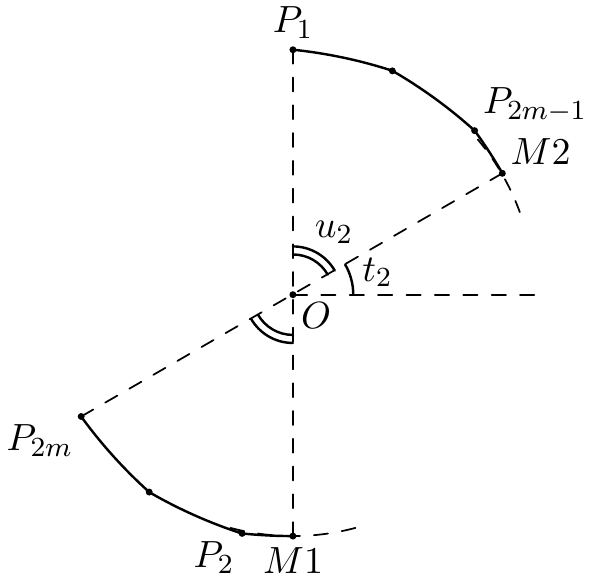} \quad \includegraphics[height=4truecm] {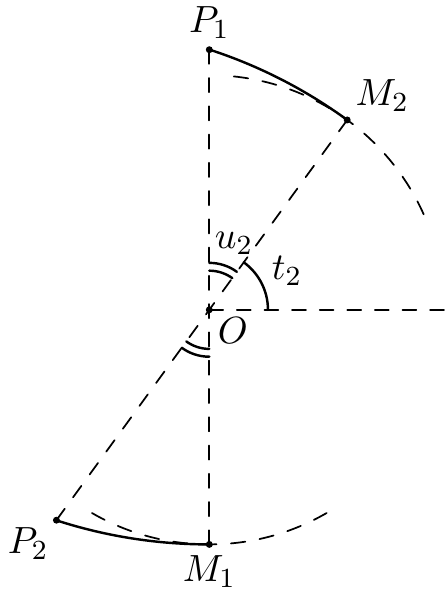}}                                                
\caption{{\it Left: the families $P_1,\ldots, P_{2m}$ and $P_2,\ldots,P_{2m}$ associated to two contact points $M_1$ and $M_2$. Right: an example with $m=1$.}}\label{fig-t}
\end{center}                                                  
\end{figure}

This leads us to define 
inside the sector $[t_2,t_1+\pi]$ the sequence of numbers $t_2<x_1<x_2< \ldots <x_{2m-2}<x_{2m-1}<t_1+\pi$ where
$$x_1=\beta_{2m},x_2=\alpha_{2m-2},x_3=\beta_{2m-2},\ldots x_{2m-2}=\alpha_2,x_{2m-1}=\beta_2.$$

As a function of the parameters $m$ and $u$, the inradius is given by the following.
\begin{lemma}\label{lem-rmt} Let $\Omega$ be a Reuleaux polygon. Let $u$ and $m$ be the two parameters of a sector, as in Definition \ref{def-sec}. Then 
\begin{equation}\label{formulainradius}
r(\Omega) = 1-\frac{\sum_{k=2}^{2m} \cos \beta_k}{\sin u} = 1- \frac{\sum_{k=1}^{2m-1} (-1)^{k-1}\cos x_k}{\sin u}.
\end{equation}
\end{lemma}
\begin{proof}
Throughout the proof, for brevity we set $r=r(\Omega)$. Without loss of generality, we may assume to work with the first sector, delimited by $M_2$ and $P_1$. Up to a rigid motion, $M_1=(0,-r)$. Accordingly, $t_1=3\pi/2$, and the sector under study is $[t_2,\pi/2]$. The statement simply follows by writing $P_1$ in two different ways: by construction, $P_1=(0,1-r)$; on the other hand, exploiting the rule $P_j= P_{j+1} + e^{i \beta_j}$ and $P_{2m}=-(1-r)e^{it_2}$, we get
$$
P_1=-(1-r) e^{i t_2} + \sum_{k=2}^{2m} e^{i \beta_k}. 
$$
Taking the projections on the horizontal and vertical components, we conclude that
\begin{equation}\label{sys}
\left\{
\begin{array}{ccc}
0=-(1-r)\cos t_2 +\sum_{k=2}^{2m} \cos \beta_k \\
\smallskip
\\
(1-r)= -(1-r)\sin t_2 +\sum_{k=2}^{2m} \sin \beta_k .
\end{array}
\right.
\end{equation}
The first line of the system, recalling that the length of the sector here is $\pi/2-t_2$, gives the first statement. The second one comes from the definition of the $x_k's$ and the relation $\alpha_k=\beta_{k+1}-\pi$.
\end{proof}
\begin{remark}
Notice that we do not require $M_2$ to be the ``first'' contact point met in the path. Moreover, notice that $m$ and $u$ do not depend on the orientation chosen.
\end{remark}

We conclude the paragraph with some estimates for the length of a sector, which will be crucial in the next section. 
\begin{lemma}\label{lengthsector}
The length $u$ of any sector satisfies
\begin{equation}\label{lensec1}
u\geq 2\left(\sqrt{1-2r}+r(2\arctan(\sqrt{4(1-r)^2-1})-\arccos\left(\frac{r}{1-r}\right)\right) 
\end{equation}
where $r$ in the inradius of the Reuleaux polygon.\\
In particular, for a Reuleaux polygon with an inradius $r\leq r_0$ (for example
an optimal Reuleaux polygon), we have
\begin{equation}\label{lensec2}
0.9926 \leq u \leq 1.1563 .
\end{equation}
\end{lemma}
\begin{proof}
We work with the first sector, delimited by $M_2$ and $P_1$. Up to a rigid motion, we may assume that $M_1=(0,-r)$, so that $P_1=(0,1-r)$. In particular, $t_1=3\pi/2$ and the sector under study is $[t_2,\pi/2]$ with length $u:=u_2=\pi/2-t_2$. The length of the boundary of $\Omega$ between $M_2$ and $P_1$ is
$$L(M_2P_1)=x_1-t_2+x_3-x_2+x_5-x_4+\ldots + x_{2m-1} - x_{2m-2} .$$
In the same way, the length of the opposite boundary from $P_2$ to $M_1$ is
$$L(P_2M_1)=x_2-x_1+x_4-x_3+\ldots +\frac{\pi}{2}-x_{2m-1}.$$
Therefore, by addition $u=\frac{\pi}{2} - t_2 = L(M_2P_1) + L(P_2M_1)$.
Now let us introduce the point $P_0$ defined as the intersection of the arc of circle
of radius 1, centered at $P_2$ with the outercircle (of radius $1-r$), see Figure \ref{fig-geod}.
\begin{figure}[h]                                             
\begin{center}                                                
{\includegraphics[height=4.5truecm] {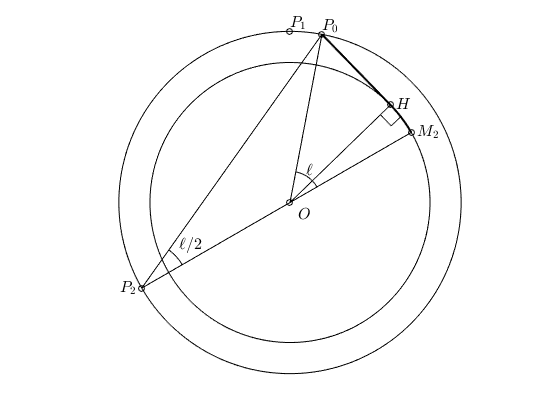}}                                                
\caption{{\it Computation of the length of a sector by comparison with a geodesic.}}\label{fig-geod}
\end{center}                                                  
\end{figure}
By convexity, the point $P_1$ is after the point $P_0$ (in the direct sense) on
the outercircle. 
We are going to make comparisons with the geodesics inside the annulus $\{r\leq |X|\leq 1-r\}$.
We denote by $geod(A,B)$ the length of the geodesic between two points $A,B$ in the annulus.
We have
$$L(M_2P_1) \geq geod(M_2,P_1) \geq geod (M_2,P_0).$$
The same inequality holds for $L(P_2M_1)$. Thus, $u\geq 2 geod(M_2,P_0)$.

Now, let us compute $geod(M_2,P_0)$.
This geodesic is made of
\begin{itemize}
\item a segment $P_0H$ joining $P_0$ to the point $H$ defining the tangent to the incircle
which goes through $P_0$ (see Figure \ref{fig-geod}),
\item the arc of the incircle \t{$M_2H$}.
\end{itemize}
By Pythagoras' theorem, $|P_0H|=\sqrt{1-2r}$.

Now we set $\ell(r)=2\arctan(\sqrt{4(1-r)^2-1})$. A simple trigonometric computation shows
that the angle $\widehat{M_2P_2P_0}$ is $\ell(r)/2$. Relations between the intercepting angles in a circle show that the angle 
$\widehat{P_0OM_2}$ is $\ell(r)$. Now the angle $\widehat{P_0OH}$ is $\arccos(r/(1-r))$, therefore
the arc \t{$M_2H$} has length which is $r(\ell(r) - \arccos(r/(1-r))$ and the inequality
\eqref{lensec1} follows.

For the inequality \eqref{lensec2}, we remark that the function in the right-hand side
of \eqref{lensec1} is decreasing and we compute its value for $r=r_0$ getting
the estimate from below. The estimate from above follows since the lengths of the three
sectors satisfy $u_1+u_2+u_3=\pi$.
\end{proof}
\begin{remark}\label{5791113} The number of points $x_k$ we can have in a sector is odd but variable. Nevertheless, we can bound this number 
 for low values of $N$. For example, for
a Reuleaux heptagon ($N=3$), there is necessarily at least one sector (two if we consider
its corresponding sector) with only one point inside. For a Reuleaux nonagon ($N=4$), either
there is one sector with only one point, or all sectors have three points.
For $N=5$ or $N=6$ there is at least one sector with either one or three points (because
if any sector has more than 5 points, we have a number of sides at least equal to 
$6$ (number of sectors) times $5$ (number of points) divided by $2 =15$).
\end{remark}

\section{Proof of the main theorem}
The key results of this section concern the maximization of $h$ in the subclass of Reuleaux polygons with a prescribed maximal number of sides, namely in $\mathcal B_N^1$, for $N\in \mathbb N$. They are:
\begin{proposition}\label{15}
A $2N+1$-Reuleaux polygon with $N\geq 7$ cannot be the maximizer.
\end{proposition}
\begin{proposition}\label{autres}
A $2N+1$-Reuleaux polygon with $2\leq N\leq 6$ cannot be the maximizer.
\end{proposition}
Once proved the propositions (see the next two paragraphs), we are done:
\begin{proof}[Proof of Theorem \ref{maintheo}]
Let $N\in \mathbb N$ be fixed. In view of Proposition \ref{existN}, $h$ admits a maximizer $\Omega_N$ in the class $\mathcal B_N^1$. Thanks to Propositions \ref{15} and \ref{autres}, $\Omega_N$ is necessarily the Reuleaux triangle, in particular $\max_{\mathcal B_N^1} h= h(\mathbb T)$. 

Let us now consider the maximization problem in the whole class $\mathcal B^1$. As already shown in Proposition \ref{exist}, the problem admits a solution. Exploiting the density with respect to the Hausdorff metric of the Reuleaux polygons (cf. \cite{BF} and \cite{Buc}), the continuity of $h$ with respect to the Hausdorff metric (see again \cite[Proposition 3.1]{Pa2}), and the fact that $\{\mathcal B^1_N\}_N$ is an increasing family with $N$, we infer that 
$$
\max_{\mathcal B^1} h =  \sup_N \max_{\mathcal B^1_N} h= \lim_{N \to \infty} \max_{\mathcal B^1_N} h.
$$ 
As shown at the beginning of the proof, the sequence $\{\max_{\mathcal B^1_N} h\}_N$ is stationary, equal to $h(\mathbb T)$. This concludes the proof.
\end{proof}

The next two paragraphs are devoted to the proof of Proposition \ref{15} and \ref{autres}, respectively.

\subsection{Reduction to Reuleaux polygons with less than 15 sides}
We start by proving that an optimal Reuleaux polygon has less than 15 sides.
\begin{theorem}\label{quinze}
Let $\Omega$ be a $2N+1$-Reuleaux polygon satisfying the optimality condition \eqref{oc3}. 
Then its inradius satisfies
\begin{equation}\label{minr}
r\geq \frac{1}{2}-\frac{h_N}{4\sin u}-\frac{1-\tau_N}{4\tau_N}\left(1+\frac{h_N^2}{6}
\frac{u}{\sin u}\right)-\frac{h_N^2}{24} \frac{u}{\sin u}
\end{equation}
where $h_N$ is the maximal length of an arc (given in Proposition \ref{propmaxlength})
and $\tau_N$ is the rate between two consecutive lengths (given in Theorem \ref{theolengths})
and $u$ is the length of a sector, as in Lemma \ref{lengthsector}.
\end{theorem}
\begin{proof}
Following the notations of Section \ref{secinradius} we consider a sector $[t_2,\pi/2]$
with $2m-1$ points $t_1<x_1<x_2< \ldots <x_{2m-1}<\frac{\pi}{2}$.
For convenience, we will denote $x_0=t_2$ and $x_{2m}=\pi/2$.
According to Formula \eqref{formulainradius}, in order to bound from below the inradius, we need to estimate from above 
$$C:=\cos x_1 -\cos x_2 + \ldots - \cos x_{2m-2} + \cos x_{2m-1}$$
that can also be written $C=\sum_{k=1}^m \int_{x_{2k-1}}^{x_{2k}} \sin t\, \de t$. We also introduce $C'=\sum_{k=0}^{m-1} \int_{x_{2k}}^{x_{2k+1}} \sin t\, \de t$ so that
$C+C'=\int_{t_2}^{\pi/2} \sin t \, \de t = \cos t_2 =\sin u$ since $u=\frac{\pi}{2}-t_2$.

The idea of the proof is to use the trapezoidal rule to estimate both integrals $C$ and $C'$
taking advantage on the information we have on the lengths of each interval. Let us denote
$h_k=x_{k+1}-x_k$ and we recall that, in view of Theorem \ref{theolengths}, for any $k$: 
\begin{equation}\label{encad}
\tau_N h_{k-1} \leq h_k \leq \frac{h_{k-1}}{\tau_N}.
\end{equation}
Let us introduce the approximations of $C$ and $C'$ obtained by the trapezoidal rule:
\begin{eqnarray*}
C_h=\sum_{k=1}^m \frac{h_{2k-1}}{2}(\sin x_{2k-1} +\sin x_{2k})\\
C'_h=\sum_{k=0}^{m-1} \frac{h_{2k}}{2}(\sin x_{2k} +\sin x_{2k+1}).
\end{eqnarray*}
The classical error formulae in numerical integration provide
\begin{eqnarray*}
- \frac{h_N^2}{12}\sum_{k=1}^m h_{2k-1} \leq C-C_h\leq \frac{h_N^2}{12}\sum_{k=1}^m h_{2k-1}\\
- \frac{h_N^2}{12}\sum_{k=0}^{m-1} h_{2k} \leq C'_h-C'\leq \frac{h_N^2}{12}\sum_{k=0}^{m-1} h_{2k}\\
\end{eqnarray*}
that yields by adding the two inequalities and using $u=x_{2m}-x_0$
\begin{equation}\label{inte1}
-\frac{u h_N^2}{12} + C_h-C'_h \leq C-C' \leq +\frac{u h_N^2}{12} + C_h-C'_h \,.
\end{equation}
We write
$$C_h-C'_h=-\frac{h_0}{2} \sin x_0 +\sum_{k=1}^{2m-1} (-1)^{k-1} \frac{h_k-h_{k-1}}{2} \sin x_k
+\frac{h_{2m-1}}{2}\,.$$
Now, using \eqref{encad}
$$|h_k - h_{k-1}| \sin x_k \leq \left(\frac{1}{\tau_N}-1\right) \min (h_{k-1},h_k) \sin x_k
\leq \left(\frac{1}{\tau_N}-1\right) \frac{h_{k-1}+h_k}{2} \sin x_k.$$
Thus
$$C_h-C'_h \leq -\frac{h_0}{2} \sin x_0 + \frac{1}{2} \left(\frac{1}{\tau_N}-1\right)
\sum_{k=1}^{2m-1} \frac{h_{k-1}+h_k}{2} \sin x_k + \frac{h_N}{2}\,.$$
We use now the middle-point integration rule to estimate the term
$\sum_{k=1}^{2m-1} \frac{h_{k-1}+h_k}{2} \sin x_k$. First with the odd points $x_{2k-1}$
on intervals of length less than $2h_N$:
$$\left|\int_{x_0}^{x_{2m}} \sin t\, \de t - \sum_{k=1}^{m} (h_{2k-1}+h_{2k-2}) \sin x_{2k-1}\right|
\leq \frac{u (2 h_N)^2}{24}$$
then with the even points $x_{2k}$:
$$\left|\int_{x_1}^{x_{2m-1}} \sin t\, \de t - \sum_{k=1}^{m-1} (h_{2k-1}+h_{2k}) \sin x_{2k}\right|
\leq \frac{u (2 h_N)^2}{24}\,.$$
Therefore by addition
$$\sum_{k=1}^{2m-1} \frac{h_{k-1}+h_k}{2} \sin x_k \leq \frac{1}{2}(\cos x_0 +\cos x_1 -\cos x_{2m-1})+ \frac{u h_N^2}{6} \leq \cos x_0 + \frac{u h_N^2}{6}$$
and we infer
\begin{equation}
C_h-C'_h \leq \frac{1}{2} \left(\frac{1}{\tau_N}-1\right)(\cos x_0 + \frac{u h_N^2}{6})
+ \frac{h_N}{2}\,.
\end{equation}
Finally, using \eqref{inte1} together with $C+C'=\sin u$, we obtain
$$
C\leq \frac{\sin u}{2}+  \frac{h_N}{4} + \frac{1}{4} \left(\frac{1}{\tau_N}-1\right)(\sin u + \frac{u h_N^2}{6}) + \frac{u h_N^2}{24},
$$
which, combined with $1-r=C/\sin u$, gives \eqref{minr}.
\end{proof}
As a corollary, we can give the 
\begin{proof}[Proof of Proposition \ref{15}]
When $N$ increases,  according to Theorem \ref{theolengths} and Proposition \ref{propmaxlength},
the maximal length $h_N$ decreases while the rate $\tau_N$ increases. Therefore, the right-hand side of inequality \eqref{minr} is increasing with $N$. In other words, if we prove that
the inradius of a $15$-Reuleaux polygon (satisfying the optimality conditions)
is greater than $r_0$, it will also be true for any $2N+1$-Reuleaux polygon 
(satisfying the optimality conditions) with $N\geq 7$. Now, according to Proposition
\ref{propr0}, this shows that these Reuleaux polygons cannot be optimal.

We have seen in Lemma \ref{lengthsector}, Formula \eqref{lensec2}, that for an optimal domain (thus with an inradius less than
$r_0$), we can choose the largest sector whose length satisfies $\frac{\pi}{3} \leq u \leq 1.1563$.
This implies in particular $u/\sin u \leq 1.2633$ and $1/\sin u \leq 2/\sqrt{3}$.
Plugging these bounds in \eqref{minr} together with $h_N=0.2194$ and $\tau_N=0.9875$
(see Table \ref{table1}) provides an inradius $r>r_0$ that gives the thesis.
\end{proof}

\subsection{The case of polygons with a number of sides between 5 and 13.}
In this paragraph we rule out the intermediate cases, corresponding to Reuleaux polygons satisfying the optimality conditions and having a number of sides between $5$ and $13$. 
\begin{proof}[Proof of Proposition \ref{autres}] Throughout the proof, we consider a Reuleaux polygon satisfying the optimality conditions. Its inradius, for brevity, will be simply denoted by $r$. In view of Proposition \ref{propr0}, it is enough to show that $r>r_0$. The proof is organized as follows: first, we treat the case of Reuleaux pentagons ($N=2$); then we analyze the Reuleaux polygons with $N=3,4,5,6$, distinguishing the cases in which there is a sector of 1 point or not. Note that the former is always satisfied for heptagons ($N=3$); moreover, in the latter, there is always a sector with 3 points (see also Remark \ref{5791113}).
\medskip

\noindent {\it Step 1. The case of pentagons, $N=2$.} As already noticed in Remark \ref{penta}, for pentagons the inradius is given by formula \eqref{formula5}. Since $\mathcal H^1(\gamma_1)\leq h_2^{max}$, we get, thanks to Table 1:
$$
r\geq 1- \frac{1}{2\cos(h_2^{max}/2)}>0.47>r_0.
$$
This concludes the proof for pentagons.
\medskip

\noindent{\it Step 2. The case of a sector with $1$ point, for $N=3,4,5,6$}. Without loss of generality, up to a rotation, we may assume that such a sector is the segment $[t,\pi/2]$ with length $u= \pi/2-t$. According to this notation, we rewrite system \eqref{sys} as follows:
$$
\left\{\begin{array}{ccc}
&(1-r)\sin u=\cos x_1 \\
&(1-r)(1+\cos u)=\sin x_1,
\end{array}
\right.
$$
so that, taking the quotient, we get
$$
\tan(u/2)= \frac{1}{\tan x_1 }.
$$
Therefore
\begin{align*}
1-r& = \frac{\cos x_1 }{\sin(u)} = \frac{\cos x_1 (1+ \tan^2(u/2))}{2 \tan (u/2)}
\\ & = \frac{\cos x_1 (1 + 1/\tan^2(x_1)) \tan x_1}{2} \\
& = \frac{1}{2\sin x_1}.
\end{align*}
The value of $x_1$ is unknown, however, its distance from $\pi/2$ is at most the maximal length of one of the arcs: $x_1 \geq \pi/2-h_N^{max}$. Since $h_N^{max}$ is decreasing with respect to $N$, we get
$$
1-r\leq \frac{1}{2\sin (\pi/2-h_N^{max})}=\frac{1}{2  \cos (h_N^{max})}\leq \frac{1}{2  \cos (h_3^{max})},
$$
so that, using Table 1:
\begin{equation}\label{estimate1} 
r \geq 1- \frac{1}{2\cos(h_3^{max})}>0.44 >r_0.
\end{equation}
This concludes the proof of the step.
 
\medskip

\noindent {\it Step 3. The case of a sector with 3 points, for $N=4,5,6$.} Let $N$ be fixed. As in Step 2, without loss of generality, up to a rotation, we may assume that such a sector is $[t,\pi/2]$ with length $u=\pi/2-t$. By assumption, there exists a sector with 3 points $x_1,x_2,x_3$.  By Lemma \ref{lem-rmt}, we have
$$
r= 1- \frac{\left( \cos x_1 -\cos x_2 + \cos x_3 \right) }{\cos t}.
$$
We claim that
\begin{equation}\label{claimt}
t\in [t_0, t_1]\subset [\pi/11,4\pi/11].
\end{equation}
Let us assume the claim true (it will be shown at the end of the proof). In order to have a lower bound for $r$ we look for an upper bound for $C:= \cos x_1 -\cos x_2 + \cos x_3 $, in $[t_0, t_1]$. Let $h_1:=x_2-x_1$ and $h_2:=x_3-x_2$. Set $h:= (h_1+h_2)/2$ and $\delta:= (h_1-h_2)/2$. Therefore
\begin{align*}
C& =\cos(x_2-h -\delta) -  \cos x_2  + \cos(x_2 + h -\delta)
\\
& = \frac12 \left( \cos(x_2-h -\delta) - \cos x_2\right)  + \frac12 \left( \cos(x_2+h -\delta) - \cos x_2 \right) 
\\  & \quad + \frac12 \left( \cos(x_2-h -\delta) +\cos(x_2+h-\delta)\right)
\\
& = \cos(x_2-\delta) (1-4\sin^2(h/2))+2\sin(x_2-\delta/2)\sin(\delta/2).
\end{align*}
Without loss of generality, up to consider the opposite sector, we may assume that 
$$
x_2\geq \frac{t+\pi/2}{2}.
$$
Moreover, there holds
$$
|\delta|\leq (1-\tau_N)\frac{h_N^{max}}{2}.
$$
These two bounds give
\begin{eqnarray*}
&x_2-\delta \geq \frac{t}{2} + \frac{\pi}{4} - (1-\tau_N) h_N^{max}/2\quad \Rightarrow \quad \cos(x_2-\delta) \leq \cos (t/2 + \pi/4 - (1-\tau_N) h_N^{max}/2)&
\\
&1-4\sin^2(h/2) \leq 1-4\sin^2(h_N^{min}/2)&
\\
&2\sin(x_2-\delta/2)\sin(\delta/2) \leq 2 \sin (\delta/2) \leq 2 \sin ((1-\tau_N)h_N^{max}/4).&
\end{eqnarray*}
Using these bounds in the expression of $C$, we obtain the following upper bound for $C/\cos t$:
\begin{align*}
\frac{C}{\cos t} & \leq \frac{\cos (t/2 +\pi/4-(1-\tau_N) h_N^{max}/2) (1-4\sin^2(h_N^{min}/2)) + 2 \sin ((1-\tau_N)h_N^{max}/4)}{\cos t}
\\
& \leq \frac{\cos (t/2 +\pi/4-(1-\tau_N) h_N^{max}/2)}{\cos t} \,  (1-4\sin^2(h_N^{min}/2)) + 2 \frac{\sin ((1-\tau_N)h_N^{max}/4)}{\cos t_1}.
\end{align*}

Let us consider the first term. We want to show that this is decreasing, namely we claim that 
$$
\forall t\in [t_0,t_1]\quad f(t):=\frac{\cos (t/2 +\alpha_N)}{\cos t}\leq f(t_0),
$$
with $\alpha_N:=\pi/4-(1-\tau_N) h_N^{max}/2$.
To this aim, we prove that $f'<0$ in $[t_0,t_1]$:
\begin{align*}
\cos^2(t) f'(t)= & =- \frac12\sin(t/2 + \alpha_N)\cos t + \sin t \cos(t/2+\alpha_N)
\\
& =\sin (t/2) \cos(t/2)\cos(t/2+\alpha_N)  + \frac12 \sin(t/2-\alpha_N)  \\
& = \sin (t/2) \cos(\alpha_N)\left(\frac12 + \cos^2(t/2)\right) - 
\cos (t/2) \sin(\alpha_N)\left(\frac12 + \sin^2(t/2)\right).
\end{align*}

To prove that
$$\sin (t/2) \cos(\alpha_N)\left(\frac12 + \cos^2(t/2)\right) <
\cos (t/2) \sin(\alpha_N)\left(\frac12 + \sin^2(t/2)\right)$$
we square both sides and setting $x=\cos^2(t/2)$, this leads to consider the polynomial
$$P(x)=x^3 -3\sin^2(\alpha_N) x^2 +\left(\frac94 \sin^2(\alpha_N) -\frac34\cos^2(\alpha_N)\right) x - \frac{\cos^2(\alpha_N)}{4},$$
and look when it is positive. Now, for the three values $(\alpha_4,\alpha_5,\alpha_6)\simeq
(0.7824,0.7832,0.7837)$ obtained from Table 1, we see that the polynomial $P(x)$ has only
one real root which is less than $0.65$. Therefore, as soon as $\cos(t/2)\geq \sqrt{0.65}$
we have $f'<0$ and this is the case for $t\in [t_0,t_1]$. Therefore, we have

\begin{equation}\label{lastestimate}
r\geq 1 -  \frac{\cos (t_0/2 +\pi/4-(1-\tau_N) h_N^{max}/2)}{\cos t_0} \,  (1-4\sin^2(h_N^{min}/2)) - 2 \frac{\sin ((1-\tau_N)h_N^{max}/4)}{\cos t_1}.
\end{equation}

Let us now prove the claim \eqref{claimt}:
\begin{itemize}
\item for $N=4$ the presence of a sector with 3 points and the absence of a sector with 1 point, occurs only when all the sectors have 3 points. If we choose the sector with maximal length, in view of Lemma \ref{lengthsector}, Formula \eqref{lensec2}, we infer that 
$$t\in [t_0,t_1] := [\pi/2-1.1563,\pi/6]\subset [\pi/11, 4\pi/11];
$$
\item for $N=5$ the presence of a sector with 3 points and the absence of a sector with 1 point, occurs only when two sector have 3 points and one sector has 5 points. In particular, there exists a sector with 3 points which attain either the maximal or the minimal length. In the first case, in view of Lemma \ref{lengthsector}, Formula \eqref{lensec2}, we may take (as for $N=4$) 
$$
[t_0,t_1] := [\pi/2-1.1538,\pi/6]\subset [\pi/11, 4\pi/11];
$$
in the second case, in view of Lemma \ref{lengthsector}, Formula \eqref{lensec2}, we may take 
$$[t_0,t_1] := [1.1563/2,\pi/2-2h_5^{min}]=[0.57815,1.0184]\subset [\pi/11, 4\pi/11];$$
\item for $N=6$, we obtain the bounds for $t$ in a different way: since between $t_2$ and $\pi/2$ by assumption we have 3 points, we infer that 
$$t\in [t_0,t_1]:= [\pi/2 - 4 h_6^{max}, \pi/2-2 h_6^{min}]= [0.5619,1.10505]\subset [\pi/11, 4\pi/11].
$$
\end{itemize}
Using these $t_0$ and $t_1$ in \eqref{lastestimate}, we get $r> 0.46$ for $N=4$, $r>0.44$ for $N=5$, and $r> 0.45$, for $N=6$. These lower bounds are all greater than $r_0$. This concludes the proof.
\end{proof}

\section{Conclusion and perspectives}\label{conc}
The Cheeger constant is known to be the first eigenvalue of the $1$-Laplacian, see \cite{KaFr}. On the other side,
according to \cite{JuLiMa}, the first eigenvalue of the $\infty$-Laplacian is nothing else than $1/r$ the inverse of the inradius.
Since the Reuleaux triangle maximizes $1/r(\Omega)$ in the class $\mathcal{B}^1$ and we have proved in this paper
that it also maximizes the Cheeger constant, a very natural question and conjecture is:

\smallskip
\noindent {\bf Conjecture : (Blaschke-Lebesgue Theorem for all eigenvalues)}  Prove that the Reuleaux triangle maximizes the first
eigenvalue of the $p$-Laplacian in the class $\mathcal{B}^1$ for all $p, 1\leq p\leq +\infty$.
\medskip

We conclude by noticing that the problem under study could have been set in a different class of shapes: the planar convex sets with prescribed {\it minimal width} or {\it thickness} (i.e., the minimal distance between two parallel lines enclosing the set). It is immediate to check that a maximizer in this class is actually a body of constant width, namely the width constraint is saturated in any direction. The very same reasoning applies to the minimization problem, replacing the thickness constraint with the diameter constraint.

\section{Appendix} For the benefit of the reader, we gather here the main properties of the function $\mathcal A$, studied in \cite{HL}. As already mentioned, the function $\mathcal A$ associates to $r$ the minimal area of a body with constant with (=1) and inradius $r$. The domain of definition of the function is the interval $[1-1/\sqrt{3}, 1/2]$, which spans all the possible inradii of the bodies of constant width 1. Among them, we highlight the inradii of regular Reuleaux polygons, by labeling them as $r_{_{2N+1}}$, being $2N+1$ the number of sides. The sequence $\{r_{_{2N+1}}\}_{N\in \mathbb N}$ is increasing and runs from $1-1/\sqrt{3}$ to $1/2$ (not attained).

The optimizer is unique and is always a Reuleaux polygon (the regular one for a ``good'' inradius) with a precise structure, that we write here below. The characterization of the optimizer allows one to compute the area quite easily, providing an explicit formula for $\mathcal A(r)$.

In \cite[Theorem 1.2]{HL} we have proved the following:
\begin{itemize}
\item If $r=\r$ for some $N\in \mathbb N$, then the optimal set of $\A(r)$ is the regular Reuleaux $(2N+1)$-gon. 
\smallskip
\item If instead $r_{_{2N-1}}<r<r_{_{2N+1}}$ for some $N\in \mathbb N$, $N\geq 2$, 
setting
$$
\ell(r):=2 \arctan\left(\sqrt{4(1-r)^2-1}\right),\quad x(r):=\frac\pi2 - \frac{2N-1}{2}\, \ell(r),
$$
the optimal set of $\A(r)$ is unique (up to rigid motions) and has the following structure:
\begin{itemize}
\item[i)] it is a Reuleaux polygon with $2N+1$ sides, all but one tangent to the incircle;
\item[ii)] the non tangent side has both endpoints on the outercircle and has length 
$$
a(r):=2\,\arcsin\Big((1-r)\sin(x(r))\Big),
$$ 
its two opposite sides have one endpoint on the outercircle and meet at a point in the interior of the annulus; moreover, they both have length 
$$
b(r):= x(r) + \frac{\ell(r) -a(r)}{2};
$$
\item[iii)] the other $2N-2$ sides are tangent to the incircle, have both endpoints on the outercircle, and have length $\ell(r)$.
\end{itemize}
\smallskip
\item Setting
\begin{align*}
A(r,x,a,b):=& (1-r)^2 \sin x \cos x + \frac{a-\sin a}{2} +b-\sin b
\\ & +(1-r) \big(\cos(a/2)-(1-r)\cos x\big) \sin(x+\ell(r)),
\end{align*}
the least area reads
$$
\mathcal A(r)=\left\{
\begin{array}{lll}
(2N+1)A(\r,0,0,0)\quad & \hbox{if }r=\r,
\\
(2N-2)A(r,0,0,0)+A(r,x(r),a(r),b(r)) \quad & \hbox{if } r_{_{2N-1}}<r<r_{_{2N+1}}.
\end{array}
\right.
$$

\smallskip

\item The function $r\mapsto \mathcal A(r)$ is continuous and increasing.
\end{itemize}

\bigskip

\noindent {\bf Acknowledgements}: 
This work was partially supported by the project ANR-18-CE40-0013 SHAPO financed by the French Agence Nationale de la Recherche (ANR). IL acknowledges the Dipartimento di Matematica - Universit\`a di Pisa for the hospitality.

\medskip

Antoine \textsc{Henrot}, Universit\'e de Lorraine CNRS, IECL, F-54000 Nancy, France. E-mail: \texttt{antoine.henrot@univ-lorraine.fr} 

Ilaria \textsc{Lucardesi}, Universit\'e de Lorraine CNRS, IECL, F-54000 Nancy, France. E-mail: \texttt{ilaria.lucardesi@univ-lorraine.fr}

\end{document}